\documentclass[10 pt, journal]{IEEEtran}	

\usepackage{amsmath,amsfonts, amssymb,amsthm}
\usepackage{commath}
\usepackage{algorithmic}
\usepackage{algorithm}
\usepackage[usenames, dvipsnames]{color}

\usepackage{epsfig} 
\usepackage{psfrag}

\newcommand{\PA}[1]{$\spadesuit$\footnote{MP: #1}}
\newcommand{\IL}[1]{$\clubsuit$\footnote{ILIJA: #1}}

\usepackage{times}
\usepackage[tight,footnotesize]{subfigure}

\newtheorem{theorem}{Theorem}
\newtheorem{corollary}{Corollary}
\newtheorem{lemma}{Lemma}
\newtheorem{remark}{Remark}
\newtheorem{definition}{Definition}

\makeatletter
\def\@copyrightspace{\relax}
\makeatother

\begin{document}

\title{Relaxing Integrity Requirements for Attack-Resilient Cyber-Physical Systems
}


\author{Ilija Jovanov, and Miroslav~Pajic,~\IEEEmembership{Member,~IEEE}
\thanks{This work was supported in part by the NSF CNS-1652544 and CNS-1505701 grants, and the Intel-NSF Partnership for Cyber-Physical Systems Security and Privacy. This material is also based on research sponsored by the ONR under agreements number N00014-17-1-2012 and N00014-17-1-2504. Some of the preliminary results have appeared in~\cite{jovanov_cdc17}.}
\thanks{I Jovanov and M. Pajic are with the Department of Electrical and Computer Engineering, Duke University, Durham,
NC, 27708 USA. E-mail:  {\tt \{ilija.jovanov, miroslav.pajic\}@duke.edu}.}
}

\maketitle
\begin{abstract}
The increase in network connectivity has also resulted in several high-profile attacks on cyber-physical systems. 
An attacker that manages to access a local network could \emph{remotely} affect control performance by tampering with sensor~measurements delivered to the controller. Recent results have shown that with network-based attacks, such as \emph{Man-in-the-Middle} attacks, the attacker can introduce an unbounded state estimation error if measurements from a suitable subset of sensors contain false data when delivered to the controller. While these attacks can be addressed with the standard cryptographic tools that ensure data integrity, their continuous use would introduce significant communication  and computation overhead. 
Consequently, we study effects of intermittent data integrity guarantees on system performance under stealthy attacks. We consider~linear estimators equipped with a general type of  residual-based intrusion detectors (including $\chi^2$ and SPRT detectors), and show that even when integrity of sensor measurements is enforced only intermittently, the attack impact is significantly limited; specifically, the state estimation error is bounded or the attacker cannot remain stealthy.
Furthermore, we present methods to: (1)~evaluate the effects of any given integrity enforcement policy in terms of reachable state-estimation errors for any type of stealthy attacks, and (2)~design an enforcement policy that provides the desired estimation error guarantees under attack. Finally, on three automotive case studies we show that even with less than 10\% of authenticated messages we can ensure satisfiable control performance in the presence of attacks. 
\end{abstract}


\section{Introduction}
\label{sec:intro}

Several high-profile incidents have recently exposed vulnerabilities of cyber-physical systems (CPS) 
and drawn attention to the challenges of providing security guarantees as part of their design. These incidents cover a wide range of application domains and system complexity, from attacks on large-scale infrastructure such as the 2016 breach of Ukrainian power-grid~\cite{ukraine_attack16}, to the StuxNet virus attack on an industrial SCADA system~\cite{stuxnet2011}, as well as attacks on controllers in modern cars (e.g.,~\cite{car_security2011}) and unmanned arial vehicles~\cite{shepard12}  

There are several reasons 
for such number of security related incidents affecting control of CPS. 
The tight interaction between information technology and physical world has greatly increased the attack vector space. 
For instance, an adversarial signal can be injected into measurements obtained from a sensor, using non-invasive attacks that modify the  sensor's physical environment; 
 as shown in attacks on GPS-based navigation systems~\cite{gps_spoof1,kerns2014unmanned}. 
Even more important reason is network connectivity that is prevalent in CPS. An attacker that manages to access a local control network could \emph{remotely} 
affect control performance by tampering with sensor measurements and actuator commands in order to force the plant into any desired state, as illustrated in~\cite{smith_decoupled_attack11}. 
From the controls perspective, attacks over an internal system network, such as the \emph{Man-in-the-Middle} (MitM) attacks where the attacker inserts messages anywhere in the sensors$\rightarrow$controllers$\rightarrow$actuators pathway, can be modeled as additional malicious signals injected into the control loop via the system's sensors and actuators~\cite{ncs_attack_models}.

While the interaction with the physical world introduces new attack surfaces, it also provides opportunities 
to improve system resilience agains attacks. 
The use of control techniques that employ a physical model of the system's dynamics for attack detection and attack-resilient state estimation has drawn significant attention in recent years (e.g.,~\cite{ncs_attack_models,teixeira2015secure2, pasqualetti2013attack,fawzi_tac14,sundaram_cdc10,mo2015physical,amin_csm15,pajic_iccps14,pajic_tcns17,shoukry2016smt}, and a recent survey~\cite{DBLP:journals/corr/LunDMB16}). 
%
One line of work is based on the use of unknown input observers (e.g.,~\cite{sundaram_cdc10,pasqualetti2013attack}) and non-convex optimization for resilient estimation (e.g.,~\cite{fawzi_tac14,pajic_tcns17}), while another focuses on attack-detection and estimation guarantees in systems with standard Kalman filter-based state estimators (e.g.,~\cite{mo2010falseincsys,mo_procIEEE12, kwon2014stealthy,kwon2015real,mo2015physical,yilin_tac16,kwon_tac17}). 
In the later works, estimation residue-based failure detectors, such as~$\chi^2$~\cite{mo2010falseincsys,yilin_tac16} and sequential probability ratio test (SPRT) detectors~\cite{kwon2015real}, are employed for intrusion detection. 
Still, irrelevant of the utilized attack detection mechanism, after compromising a suitable subset of sensors, an intelligent attacker can
significantly degrade control performance while remaining undetected (i.e.,~stealthy). 
%
%
For instance, for resilient state estimation techniques as in~\cite{fawzi_tac14,pajic_tcns17}, measurements from at least half of the sensors should not be tampered with~\cite{fawzi_tac14,shoukry2013event},
while~\cite{mo2010falseincsys, kwon2014stealthy} capture attack requirements for Kalman filter-based estimators. 
The reason for such conservative results lies in the common initial assumption that once a sensor or its communication to the estimator is compromised, all values received from the sensors can be potentially corrupted -- i.e.,~integrity of the data received from these sensors cannot be guaranteed. 

On the other hand, most of network-based attacks, including MitM attacks, can be avoided with the use of standard cryptographic tools. For example, to authenticate data and ensure integrity of received communication packets, a common approach is to add a message authentication code (MAC) to the transmitted sensor measurements. Therefore, data integrity requirements can be imposed by the continuous use of  MACs in all transmissions from a sufficient subset of sensors.
However, the overhead caused by the continuous computation and communication of authentication codes can limit their use.  
For instance,
 adding MAC bits to networked control systems that employ Controller Area Networks (CAN) may not be feasible due to the message length limitation (e.g.,~only 64 payload bits per packet in the basic CAN protocol), while splitting them into several communication packets significantly increases the message transmission time~\cite{lin2015security_tecs}. 
To illustrate this, 
consider two sensors periodically transmitting measurements over a shared network. As presented in \figref{fig:Sched}(a), without authentication (i.e.,~if transmitted data contain no MAC bits) the communication packets will be schedulable but the system would be vulnerable to false-data injection attacks. Yet, if all measurements from both sensors are authenticated, 
with the increase in the packet size due to authentication overhead, it is not possible to schedule transmissions from both sensors  in every communication frame (\figref{fig:Sched}(b)). 
Finally, a feasible schedule exists if MAC bits are attached to every other measurement packet transmitted by each sensor (\figref{fig:Sched}(c)).

\begin{figure}
\centering
\includegraphics[width=0.5\textwidth]{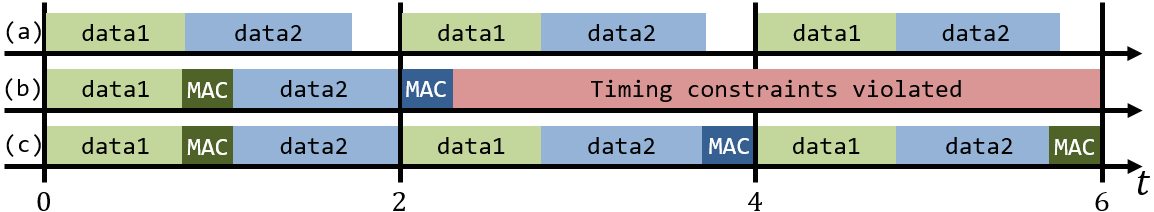}
\caption{Communication schedule for periodic messages (with period $T_s=2$) from two sensors over a shared network: (a) a feasible schedule for non-authenticated messages (i.e.,~when MAC bits are not attached to transmitted packets);  
(b) there is no feasible schedule when all messages are authenticated; and (c) if data integrity is only intermittently enforced (e.g.,~by adding MAC bits only to every other packet), scheduling of the messages becomes feasible.} 
\label{fig:Sched}%
\end{figure}




Consequently, in this paper we focus on state estimation in systems with {\emph{intermittent data integrity guarantees}} for sensor measurements delivered to the estimator. Specifically, we study the performance of linear filters equipped with 
residual-based intrusion detectors 
in the presence of attacks on sensor measurements. 
We build on the system model from~\cite{mo2010falseincsys, kwon2014stealthy,yilin_tac16} by capturing that the use of authentication mechanisms in intermittent time-points ensures that sensor measurements received at these points are valid. 
To keep our discussion and results general, we consider a wide class of detection functions that encompasses commonly used detectors, including~$\chi^2$ and SPRT detectors. 
%
%
We show that even when integrity of communicated sensor data is enforced only intermittently and the attacker is fully aware of the times of the enforcement, the attack impact gets significantly limited; concretely, either the state estimation error remains bounded or the attacker cannot remain stealthy. This holds even when communication from \emph{all} sensors to the estimator can be compromised as well as in any other case where otherwise (i.e.,~without integrity enforcements) an unbounded estimation error can~be~introduced.

Furthermore, to facilitate the use of intermittent data integrity enforcement for control of CPS in the presence of network-based attacks, we introduce an analysis and design framework that addresses two challenges. First, we introduce techniques to evaluate the effects of any given integrity enforcement policy in terms of reachable state-estimation errors for any type of stealthy attacks.  
Note that methods to evaluate potential state estimation errors due to attacks are considered in~\cite{yilin_tac16,kwon2015real,mo2010falseincsys}. However, given that the previous work considers system architectures without intermittent use of authentication, these techniques result in overly conservative estimates of reachable regions or they cannot capture the effects of intermittent integrity guarantees on the estimation error. 
Second, we present a method to design an enforcement policy that provides the desired estimation error guarantees for any attack signal inserted via compromised sensors. 
The developed framework also facilitates tradeoff analysis between the allowed estimation error and the rate at which data integrity should be enforced -- i.e.,~the required system resources such as communication~bandwidth as we have presented in~\cite{lesi_rtss17}. 

The rest of the paper is organized as follows. In Section~\ref{sec:problem}, we introduce the problem, including the system and attack models. 
In Section~\ref{sec:impact_attacks}, we analyze the impact of stealthy attacks in systems without integrity enforcements and formally define intermittent 
integrity enforcement policies. Section~\ref{sec:se_wIntegrity} focuses on state estimation guarantees  when data integrity is at least intermittently enforced. We then introduce a methodology to analyze effects of integrity enforcement policies as well as design suitable policies that ensure the desired estimation error even in the presence of attacks (Section~\ref{sec:framework}). 
Finally, in Section~\ref{sec:case_studies}, we present case studies that illustrate effectiveness of our approach, before providing final remarks in Section~\ref{sec:conclusion}. 

\subsection{Notation and Terminology}
\label{sec:notation}


%
The transpose of matrix $\mathbf{A}$ is specified as $\mathbf{A}^T$, while the $i^{th}$ element of a vector $\mathbf{x}_k$ is denoted by $\mathbf{x}_{k,i}$. 
Moore-Penrose pseudoinverse of matrix $\mathbf{A}$ is denoted as $\mathbf{A}^{\dagger}$. In addition, $\|\mathbf{A}\|_i$ denotes the $i$-norm of a matrix $\mathbf{A}$ and, for a positive definite matrix $\mathbf{Q}$, $\|\Delta\mathbf{z}_k\|_{\mathbf{Q}^{-1}}=\|\mathbf{Q}^{-1/2}\Delta\mathbf{z}_k\|_2$. 
$null(\mathbf{A})$ denotes the null space of the matrix. 
Also, $\hbox{diag}\left(\cdot\right)$ indicates a square matrix with the quantities inside the brackets on the diagonal, and zeros elsewhere, 
while $BlckDiag\left(\cdot\right)$ denotes a block-diagonal operator. 
We denote positive definite and positive semidefinite matrix $\mathbf{A}$ as $\mathbf{A}\succ 0$ and $\mathbf{A}\succcurlyeq 0$, respectively, 
while $det(\mathbf{A})$ stands for the determinant of the matrix.
Also, $\mathbf{I}_p$ denotes the $p$-dimensional identity matrix, and $\mathbf{0}_{p\times q}$ denotes $p\times q$ matrix of zeroes. 
We use $\mathbb{R}, \mathbb{N}$ and $\mathbb{N}_0$ to denote the sets of reals, natural numbers and nonnegative integers, respectively.
 {As most of our analysis considers bounded-input systems, we refer to any eigenvalue $\lambda$ 
 as \emph{unstable eigenvalue} if $|\lambda |\geq 1$. 
} 

For a set $\mathcal{S}$, we use $|\mathcal{S}|$ to denote the cardinality (i.e.,~size) of the set.  
In addition, for a set $\mathcal{K}\subset\mathcal{S}$, with $\mathcal{K}^{\complement}$ we denote the complement set of $\mathcal{K}$ with respect to $\mathcal{S}$ -- i.e., $\mathcal{K}^{\complement}=\mathcal{S}\setminus\mathcal{K}$. 
%
{Projection vector $\mathbf{i}_{j}^T$ denotes the row vector (of the appropriate size) with a 1 in its $j^{th}$ position being the only nonzero element of the vector.} For a vector $\mathbf{y}\in\mathbb{R}^p$, we use $\mathbf{P}_{\mathcal{K}}\mathbf{y}$ to denote the projection from the set $\mathcal{S}=\{1,...,p\}$ to set $\mathcal{K}$ ($\mathcal{K}\subseteq\mathcal{S}$) by keeping only elements of $\mathbf{y}$ with indices from~$\mathcal{K}$.\footnote{Formally, $\mathbf{P}_{\mathcal{K}}=\left[\begin{smallmatrix}
\mathbf{i}_{k_1} | \ldots | \mathbf{i}_{k_{|\mathcal{K}|}}\end{smallmatrix}\right]^T$, where $\mathcal{K}=\{s_{k_1},...,s_{k_{|\mathcal{K}|}}\}\subseteq{S}$ and $k_1<k_2<...<k_{|\mathcal{K}|}$.}
Finally, \emph{the support} of the vector $\mathbf{v}\in\mathbb{R}^p$ is the set 
$\hbox{supp}(\mathbf{v})=\{i ~|~ \mathbf{v}_i\neq 0\}\subseteq\{1,2, ..., p\}.$ 








\section{Problem Description}
\label{sec:problem}

Before introducing the problem formulation, we describe the considered system and its architecture (shown in \figref{fig:FullSystem}), as well as the attacker model.

\subsection{System Model without Attacks}
\label{sec:sys_model}

We consider an observable linear-time invariant (LTI) system 
whose evolution without attacks can be represented~as
\begin{equation}
	\begin{split}
		\mathbf{x}_{k+1}&=\mathbf{A}\mathbf{x}_{k}+\mathbf{B}\mathbf{u}_{k}+\mathbf{w}_{k}\\
		\mathbf{y}_{k}&=\mathbf{C}\mathbf{x}_{k}+\mathbf{v}_{k}
		\label{eq:ModlStSpc}
	\end{split}	
\end{equation}
where $\mathbf{x}_{k}\in\mathbb{R}^n$ and $\mathbf{u}_{k}\in\mathbb{R}^m$ denote the plant's state and input vectors, at time $k$, while the plant's output vector $\mathbf{y}_k\in\mathbb{R}^p$  contains measurements provided by $p$ sensors from the set $\mathcal{S}=\{s_1,s_2, ...,s_p\}$. Accordingly, the matrices 
$\mathbf{A}, \mathbf{B}$ and $\mathbf{C}$ have suitable dimensions. Also, $\mathbf{w}\in\mathbb{R}^n$ and $\mathbf{v}\in\mathbb{R}^p$ denote the process and measurement noise; we assume that $\mathbf{x}_0$, $\mathbf{w}_{k}$, and $\mathbf{v}_{k}$ are independent Gaussian random variables.

\begin{figure}[!t]%
\centering
\includegraphics[width=0.46\textwidth]{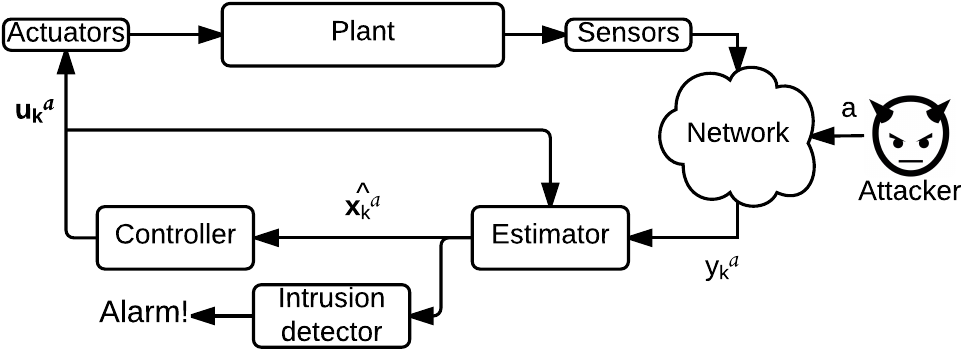}
\caption{System architecture --  by launching \emph{Man-in-the-Middle} (MitM) attacks, the attacker can inject adversarial signals into plant measurements obtained from system sensors.} 
\label{fig:FullSystem}%
\end{figure}

Furthermore, the system is equipped with an estimator in the form of a Kalman filter. Given that the Kalman gain usually converges in only a few steps, to simplify the notation we assume that the system is in steady state before the attack. 
Hence, the Kalman filter estimate $\hat{\mathbf{x}}_{k}$ is updated as
\begin{equation}
	\hat{\mathbf{x}}_{k+1}=\mathbf{A}\hat{\mathbf{x}}_{k}+\mathbf{Bu}_{k}+\mathbf{K}(\mathbf{y}_{k+1}-\mathbf{C}(\mathbf{A}\hat{\mathbf{x}}_{k}+\mathbf{Bu}_{k}))
\label{eq:KalmanSteady}
\end{equation}
\begin{equation}
	\mathbf{K}= \mathbf{\Sigma}\mathbf{C}^T(\mathbf{C}\mathbf{\Sigma}\mathbf{C}^T+\mathbf{R})^{-1},
\label{eq:KalmanKP}
\end{equation}
where $\mathbf{\Sigma}$ is the estimation error covariance matrix, and $\mathbf{R}$ is the sensor noise covariance matrix.
Also, the residue $\mathbf{z}_{k}\in\mathbb{R}^p$ at time $k$ and its covariance matrix $\mathbf{Q}$ are defined as
%
\begin{equation}
	\begin{split}
		\mathbf{z}_{k} &= \mathbf{y}_{k} - \mathbf{C}(\mathbf{A}\hat{\mathbf{x}}_{k-1}+\mathbf{Bu}_{k-1}),\\
		\mathbf{Q} &= E\{\mathbf{z}_{k}\mathbf{z}^T_{k}\} = \mathbf{C\Sigma C}^T+\mathbf{R}.
	\end{split}
\label{eq:Residue}
\end{equation}
Finally, the state estimation error $\mathbf{e}_{k}$ is defined as the difference between the plant's state $\mathbf{x}_{k}$ and Kalman filter estimate $\hat{\mathbf{x}}_{k}$ as
\begin{equation}
	\mathbf{e}_{k}= \mathbf{x}_{k}-\hat{\mathbf{x}}_{k}.
\label{eq:EstError}
\end{equation}

In addition to the estimator, we assume that the system is equipped with an intrusion detector.
We consider a general case where the detection function $g_k$ of the intrusion detector is defined as
\begin{equation}
	g_k = \sum_{i=k-\mathcal{T}+1}^{k}c_{(i-k+\mathcal{T})}{\mathbf{z}_i}^T\mathbf{Q}^{-1}\mathbf{z}_i.
\label{eq:GeneralGk}
\end{equation}
Here, $\mathcal{T}$ is the length of the detector's time window, and $c_i$ for $ i=1,...,\mathcal{T}$ are predefined non-negative coefficients, with $c_\mathcal{T}$ being strictly positive. 
The above formulation captures both fixed window size detectors, where  $\mathcal{T}$ is a constant, as well as detectors where the time window size $\mathcal{T}$ satisfies $\mathcal{T}=k$.
Also,  the definition of the detection function $g_k$  covers a wide variety of commonly used intrusion detectors, such as $\chi^2$ and sequential probability ratio test (SPRT) detectors previously considered in these scenarios~\cite{mo2010false,miao_cdc13,yilin_tac16,miao_tcns17,kwon2014stealthy,kwon2015real}.
The alarm is triggered when the value of the detection function $g_k$ satisfies that 
\begin{equation}g_k>threshold,
\label{eqn:threshold}
\end{equation}
and the probability of the alarm at time $k$ can be captured as
\begin{equation}
	\beta_k = P(g_k>threshold).
\label{eq:AlarmProb}
\end{equation}


\subsection{Attack Model}
\label{sec:attack_model}

We assume that the attacker is capable of launching MitM attacks on communication channels between a subset of the plant's sensors $\mathcal{K}\subseteq\mathcal{S}$ and the estimator; 
for instance,~by secretly relaying corresponding altered communication packets. 
However, we do not assume that the set $\mathcal{K}$ is known to the system or system designers. 
Thus, to capture the attacker's impact on the system, the system model from~\eqref{eq:ModlStSpc} becomes 
\begin{equation}
\begin{split}
	\mathbf{x}_{k+1}^a &= \mathbf{A}\mathbf{x}_{k}^a+\mathbf{B}\mathbf{u}_{k}^a+\mathbf{w}_{k}\\
	\mathbf{y}_{k}^a &= \mathbf{C}\mathbf{x}_{k}^a+\mathbf{v}_{k}+ \mathbf{a}_k.
\label{eq:BadSens}
\end{split}
\end{equation}
Here, $\mathbf{x}^a_{k}$ and $\mathbf{y}^a_{k}$ denote  the state and plant outputs in the presence of attacks, from the perspective of the estimator, since in the general case they differ from the plant's state and outputs of the non-compromised system. In addition, $\mathbf{a}_k$ denotes the signals injected by the attacker at time $k$ starting from $k=1$ (i.e.,~$\mathbf{a}_0=\mathbf{0}$);\footnote{More details about why the attacker does not insert attack at step $k=0$ can be found in Remark~\ref{rem:k0}.}  
to model MitM attacks on communication between the sensors from set $\mathcal{K}$ and the estimator, we assume that  $\mathbf{a}_k$ 
is a sparse vector from $\mathbb{R}^p$ with support in the set~$\mathcal{K}$ 
 -- i.e., 
 $\mathbf{a}_{k,i}=0$ for all $i\in\mathcal{K}^C$ and $k > 0$.\footnote{Although a sensor itself may not be directly compromised with MitM attacks, but rather communication between the sensor and~estimator, we will also refer to these sensors are \emph{compromised~sensors}. In addition, in this work we sometimes abuse the notation by using $\mathcal{K}$ to denote both the set of compromised sensors and the set of indices of the compromised sensors.} 

We consider the following \textbf{\emph{threat model}}.

\vspace{2pt}
\noindent(1) The attacker has full knowledge of the system -- in addition to knowing the dynamical model of the plant, employed Kalman filter, and detector, 
the attacker is aware of all potential security mechanism used in communication.
Specifically, we consider systems that use standard methods for message authentication to ensure data integrity, and assume that the attacker is aware at which time points data integrity will be enforced. Thus, to avoid being detected, the attacker will not launch attacks in these steps and will also take into account these integrity enforcements in planning its attacks (as described in Section~\ref{sec:impact_attacks}).\footnote{In Section~\ref{sec:se_wIntegrity}, we will also consider the case where the attacker has limited knowledge of the system's use of security mechanisms.} 
Since we model our system such that attacks start at $k=1$, this further implies that 
at $k=1$ data integrity is not enforced, as otherwise the attacker would not be able to insert false data. 

\vspace{2pt}
\noindent(2) The attacker has the required computation power to calculate suitable attack signals, while planning ahead as needed. (S)he also has the ability to inject \emph{any} signal using  communication packets mimicking sensors from the set $\mathcal{K}$, except at times when  
data integrity is enforced. For instance, when MACs are used to ensure data integrity and authenticity of 
communication packets, 
our assumption is that the attacker does not know the shared secret key used to generate the~MACs.

\vspace{4pt}
\textbf{\emph{The goal of the attacker}} is to design attack signal $\mathbf{a}_k$ such that it \emph{maximizes the error of state estimation} while ensuring that \emph{the attack remains stealthy}. To formally capture this objective and the stealthiness constraint, we denote the state estimation, residue, and estimation error of the compromised system by $\hat{\mathbf{x}}^a_{k}$, $\mathbf{z}^a_{k}$, and $\mathbf{e}^a_{k}$, respectively. Thus, the attacker's aim is to maximize $\mathbf{e}_k^a$, while ensuring that the increase in the probability of alarm 
is not significant. 
We  also define as
$$\Delta\mathbf{e}_{k}=\mathbf{e}^a_{k}-\mathbf{e}_{k}, \qquad \Delta\mathbf{z}_{k}=\mathbf{z}^a_{k}-\mathbf{z}_{k},$$ 
the change in the estimation error and residue, respectively, caused by the attacks. 
From~\eqref{eq:ModlStSpc} and~\eqref{eq:BadSens}, the evolution of these signals can be captured as { a dynamical system ${\Xi}$ of the form}
\begin{align}
\label{eq:dEk}
\Delta \mathbf{e}_{k+1} &= (\mathbf{A}-\mathbf{KCA})\Delta \mathbf{e}_{k} - \mathbf{K a}_{k+1}, \\
\label{eq:dZk}
\Delta \mathbf{z}_{k} &= \mathbf{CA}\Delta \mathbf{e}_{k-1}+\mathbf{ a}_{k},
\end{align}
with $\Delta \mathbf{e}_0 = 0$.
\begin{remark}
From the above equations, the first attack vector to affect the change in estimation error is $\mathbf{a}_1$. Thus, without loss of generality, we assume that the attack starts at $k=1$ (i.e.,~$\mathbf{a}_i = \mathbf{0}, for~all~i\leq 0$). This also implies that $\Delta \mathbf{z}_0 = 0$. 
\label{rem:k0}
\end{remark}
Note that the above dynamical system is noiseless (and deterministic), with input~$\mathbf{a}_k$ controlled by the attacker. 
Therefore, since $E[\mathbf{e}_{k}]=\mathbf{0}$  for the non-compromised system in steady state, it follows that
\begin{equation}
	\Delta\mathbf{e}_k=E[\Delta\mathbf{e}_k]=E[\mathbf{e}^a_{k}].
\label{eq:ExpDeltaEq}
\end{equation}
Given that $\Delta\mathbf{e}_k$ provides expectation of the state estimation error under the attack, this signal can be used to evaluate the impact that the attacker has on the system.\footnote{For this reason, and to simplify our presentation, in the rest of the paper we will sometimes refer to $\Delta\mathbf{e}_{k}$ as \emph{the (expected) state estimation error} instead of \emph{the change of the state estimation error} caused by attacks.} Thus, we specify the objective of the attacker as to \textbf{\textit{maximize the expected state estimation error}} (e.g., $\|\Delta \mathbf{e}_k\|_2$). 
%
This is additionally justified by the fact that 
since $\mathbf{a}_k$ is controlled by the attacker (i.e.,~deterministic to simplify of our presentation), which~implies
\begin{equation}
	Cov(\mathbf{e}^a_{k})=Cov(\mathbf{e}_{k})=\mathbf{\Sigma}.
\label{eq:eVariability}
\end{equation}

To capture the attacker's stealthiness requirements, we use the probability of alarm in the presence of an attack 
\begin{align}
\label{eqn:beta_ka}
\beta_k^a &= P(g_k^a>threshold), ~~\hbox{where}\\ 
g_k^a &= \sum_{i=k-\mathcal{T}+1}^{k}c_{(i-k+\mathcal{T})}{\mathbf{z}_i^a}^T\mathbf{Q}^{-1}\mathbf{z}_i^a.
\label{eq:AttackedGk}
\end{align}
Therefore, to ensure that attacks remain stealthy, the attacker's \emph{stealthiness constraint} in each step $k$ is to maintain 
\begin{equation}
\label{eqn:beta_stealthy}
\beta_k^a\leq\beta_k+\varepsilon,
\end{equation}
for a small predefined value of $\varepsilon>0$. 

\subsection{Problem Formulation}
\label{sec:problem_formulation}

As we will present in the next section, 
for a large class of systems, 
a stealthy attacker can easily introduce an unbounded state estimation error 
by compromising communication between some of the sensors and the estimator.
On the other hand, existing communication protocols commonly incorporate security mechanisms (e.g.,~MAC) that can ensure integrity of delivered sensor measurements. 
Specifically, this means that  the system could enforce  $\mathbf{a}_{k,i} = {0}$ for some sensor $s_i$, or $\mathbf{a}_{k} = \mathbf{0}$ if integrity for all transmitted sensor measurements is enforced at some time-step $k$. 
However, as we previously described, the integrity enforcement comes at additional communication and computation cost, effectively preventing their continuous use in resource constrained CPS.

Consequently, we focus on the problem of evaluating the impact of stealthy attacks in systems with intermittent (i.e.,~occasional) use of data integrity enforcement mechanisms.\footnote{Formal definition of such policies are presented in the next section.} Specifically, we will address the following problems:
\begin{itemize}
%
\item Can the attacker introduce unbounded state estimation errors in systems with intermittent integrity guarantees?

\item How to efficiently evaluate the impact of intermittent integrity enforcement policies on the induced state estimation errors in the presence of a stealthy attacker?

\item How to design a 
non-overly conservative development framework that incorporates guarantees for estimation degradation under attacks into design of suitable integrity enforcement policies? 
%
\end{itemize}

\section{Impact of Stealthy Attacks on State Estimation Error}
\label{sec:impact_attacks} 

To capture the impact of stealthy attacks on the system, we start with the following definition.

\begin{definition}
\label{def:attkSet}
	The set of all stealthy attacks up to time $k$ is
	\begin{equation}
		\mathcal{A}_k = \{ \mathbf{a}_{1..k} | \beta_{k'}^a\leq\beta_{k'}+\varepsilon,~\forall k', 1\leq k'\leq k \},
	\label{eq:FeasibleAttacks}
	\end{equation}
\end{definition}
\noindent where $\mathbf{a}_{1..k} = [\mathbf{a}_1^T \dots \mathbf{a}_k^T]^T$. 

When reasoning about a set of reachable state estimation errors $\mathbf{e}^a_{k}$ due to stealthy attacks from $\mathcal{A}_k$, we have to also take into account the variability of the estimation error.
From~\eqref{eq:eVariability}, 
we can define a specific region that will contain the  error~$\mathbf{e}^a_k$ with a desired probability. 
Therefore, we introduce the following definition.
\begin{definition}
\label{def:OriginalRegionR}
The $k$-reachable region $\mathcal{R}_k$ of the state estimation error under the attack (i.e.,~$\mathbf{e}^a_k$) is the set 
%
\begin{equation}
\label{eqn:Rk}
\mathcal{R}_k=\left\{ \begin{array}{c|c} \mathbf{e} \in\mathbb{R}^n &  \left.\begin{array}{c} \mathbf{e}\mathbf{e}^T\preccurlyeq E[\mathbf{e}_k^a]E[\mathbf{e}_k^a]^T + \gamma Cov(\mathbf{e}_k^a),
\\~\mathbf{e}_k^a=\mathbf{e}_k^a(\mathbf{a}_{1..k}),~\mathbf{a}_{1..k}\in\mathcal{A}_k
\end{array}\right.\end{array}\right\}.
\end{equation} 
Furthermore, the
global reachable region $\mathcal{R}$ of the state estimation error $\mathbf{e}^a_k$ is the set
%
\begin{equation}
\label{eqn:R}
\mathcal{R} = \bigcup_{{k=1}}^\infty \mathcal{R}_k.
\end{equation}
\noindent 
\end{definition}

Here,  $\gamma$ is a design parameter directly related to the desired confidence that $\mathbf{e}_k^a$ belongs to the reachable region. 
Effectively, the set $\mathcal{R}_k$ captures the set of state estimation errors that can be reached in $k^{th}$ step due to the injected malicious signal, 
while $\mathcal{R}$ captures the set of all reachable state estimation errors. 
To assess vulnerability of the system, a critical characteristic of $\mathcal{R}$ is \emph{boundedness} -- whether a stealthy attacker can introduce unbounded estimation errors. 
To simplify the boundedness analysis of $\mathcal{R}$, we start with the following theorem. 
%
\begin{theorem}
Let $g_k=\mathbf{z}_k^T\mathbf{Q}^{-1}\mathbf{z}_k$ be the detector function. Then, for any $\varepsilon>0$, 
such that $\varepsilon\leq1-\beta_k$, 
there exists a unique $\alpha >0$ such that $\beta_{k}^a\leq\beta_{k}+\varepsilon$ if and only if $\|\Delta\mathbf{z}_k\|_{\mathbf{Q}^{-1}} \leq\alpha$.
\label{thm:equConds}
\end{theorem}
\begin{proof}
In the case without attacks, in steady-state $g_k$ has $\chi^2$ distribution with $p$ degrees of freedom, since the residue $\mathbf{z}_k$ is zero-mean ($E[\mathbf{z}_k]=\mathbf{0}$) with covariance matrix~$\mathbf{Q} = \mathbf{C\Sigma C}^T+\mathbf{R}$~{\cite{Juang:1994:ASI:174720,johnson1995continuous}}. 
Furthermore, 
from~\eqref{eq:dEk} and~\eqref{eq:dZk}, $\Delta\mathbf{z}_k=\mathbf{z}_k^a  -\mathbf{z}_k$, is output of a deterministic system controlled by $\mathbf{a}_{1..k}$, and thus $\mathbf{z}_k^a$ is a non-zero mean with covariance matrix  $\mathbf{Q}$ -- i.e.,~ the attacker is only influencing the $\Delta\mathbf{z}_k = E[\mathbf{z}_k^a - \mathbf{z}_k] = E[\mathbf{z}_k^a]$. 
Therefore, $\mathbf{g}_k^a = {\mathbf{z}_k^a}^T\mathbf{Q}^{-1}\mathbf{z}_k^a$ will have a non-central $\chi^2$ distribution with $p$ degrees of freedom; the non-centrality parameter of this distribution will be $\lambda = \|\Delta\mathbf{z}_k\|_{\mathbf{Q}^{-1}}^2$~\cite{johnson1995continuous}. 

Let $h$ be the threshold for the detector 
in~\eqref{eqn:threshold}. The alarm probabilities $\beta_k = 1-P(g_k\leq h)$ and $\beta_k^a = 1-P(g_k^a\leq h)$ 
can be computed from the 
distributions for $g_k$ and $g_k^a$ as
	$$ \beta_k = 1 - F_{\chi^2}(h,p),~~~~~ \beta_k^a = 1 - F_{nc\chi^2}(h;p,\lambda), $$
where $F_{\chi^2}(h,p)$ and $F_{nc\chi^2}(h;p,\lambda)$ are cumulative distribution functions of $\chi^2$ and noncentralized $\chi^2$, respectfully, at $h$, with $p$ degrees of freedom and noncentrality parameter $\lambda$.
Since $p$ and $h$ are fixed by the system design, it follows that $\beta_k$ will be a constant, and $\beta_k^a$ will be a function of $\lambda$.

Consider $\varepsilon = \beta_k^a-\beta_k$. This means that 
\begin{equation}
\label{eqn:c1}
\varepsilon = 1 - F_{nc\chi^2}(h;p,\lambda) - \beta_k.
\end{equation}
The probability distribution function of non-central $\chi^2$ distribution is smooth (thus making $F_{nc\chi^2}(h;p,\lambda)$ smooth), and $F_{nc\chi^2}(h;p,\lambda)$ is a decreasing function of $\lambda$ \cite{johnson1995continuous}. Hence, it follows that for any $\varepsilon$ there will exist exactly one $\sqrt{\lambda}=\|\Delta\mathbf{z}_k\|_{\mathbf{Q}^{-1}}=\alpha$ such that~\eqref{eqn:c1} is satisfied. Furthermore, 
for any $\varepsilon'$ that is lower than $\varepsilon$, the corresponding $\sqrt{\lambda'}=\|\Delta\mathbf{z}_k\|_{\mathbf{Q}^{-1}}$ from~\eqref{eqn:c1} has to be lower than $\alpha$, and vice versa, which concludes the proof.
\end{proof}

Since the bound $\alpha$ for  $\|\Delta\mathbf{z}_k\|_{\mathbf{Q}^{-1}}$ in Theorem~\ref{thm:equConds} depends on 
$\varepsilon$, $h$ and the fact that the $\chi^2$ detector with $p$ degrees of freedom is used, we will denote such value as $\alpha=\alpha_{\chi^2}(\varepsilon,p,h)$.

\begin{remark}
Related results from~\cite{mo2010false,yilin_tac16}, focus only on the detection function $g_k=\mathbf{z}_k^T\mathbf{Q}^{-1}\mathbf{z}_k$ and show only sufficient conditions for stealthy attacks -- i.e.,~that in this case from a robustness condition $\|\Delta\mathbf{z}_k\|_{\mathbf{Q}^{-1}} \leq\alpha$ it follows that the the stealthiness condition $\beta_k^a \leq \beta + \varepsilon$ 
is satisfied. 
However, 
the equivalence between conditions $\|\Delta\mathbf{z}_k\|_{\mathbf{Q}^{-1}}\leq\alpha$ and $\beta_k^a \leq \beta + \varepsilon$ 
will enable us to 
reduce conservativeness of our analysis as well as 
analyze boundness of the reachability region for the general type of detection functions  from~\eqref{eq:AttackedGk}, by allowing us to 
employ both conditions interchangeably.
\end{remark}

From Definition~\ref{def:attkSet} and Theorem~\ref{thm:equConds} the following result~holds. 

\begin{corollary}
	For the detection function $g_k = \mathbf{z}_k^T\mathbf{Q}^{-1}\mathbf{z}_k$, there exists $\alpha >0$ such that the set of all stealthy attacks~satisfies
	\begin{equation}
		\mathcal{A}_k = \{ \mathbf{a}_{1..k} | \|\Delta\mathbf{z}_{k'}\|_{\mathbf{Q}^{-1}}\leq\alpha,~\forall k', 1\leq k'\leq k \}.
	\label{eq:FeasibleAttacksCor}
\end{equation}
\label{cor:AttkSet}
\vspace{-10pt}
\end{corollary}
%

The previous results introduce an equivalent `robustness-based' representation for the set of stealthy attacks in systems where $\chi^2$ detectors are used. 
They also provide a foundation to consider the more general formulation~\eqref{eq:AttackedGk} for the detector function. 
We start with the following results characterizing over- and under-approximations of the set~$\mathcal{A}_k$ in such case,  {also using suitable 
'robustness-based' representations of the stealthiness condition. By showing that reachable estimation error regions are bounded for these sets of attacks, we will be able to reason whether the reachable region of state estimation errors is bounded for attacks from the set~$\mathcal{A}_k$.}

\begin{lemma}
For a system with the detector function $g_k^a$ of the form from~\eqref{eq:AttackedGk}, the set of all stealthy attacks 
$\mathcal{A}_k$ can be underapproximated by the set
	\begin{equation}
	\underline{\mathcal{A}}_k=\{ \mathbf{a}_{1..k} | \|\Delta\mathbf{z}_{k'}\|_{\mathbf{Q}^{-1}}\leq\underline{\alpha},~\forall k', 1\leq k'\leq k \}
	\label{eq:UappA}
	\end{equation}
	(i.e.,~$\underline{\mathcal{A}}_k\subseteq \mathcal{A}_k$), where $\underline{\alpha} = \alpha_{\chi^2}(\varepsilon,\mathcal{T}p,h/c_{max})/\sqrt{\mathcal{T}}$.
\label{lem:underapp}
\end{lemma}
In essence, the lemma states that if $\|\Delta\mathbf{z}_{k}\|_{\mathbf{Q}^{-1}} \leq\underline{\alpha}$ holds, then $g_k^a\leq h$ 
for the general detection function from~\eqref{eq:AttackedGk} is satisfied with probability 
that is lower than or equal to $\beta_{k}+\varepsilon$.
%
\begin{proof}
Consider an attack sequence  $\mathbf{a}_{1..k}\in \underline{\mathcal{A}}_k$ and the resulting evolution of the system from~\eqref{eq:dEk} and~\eqref{eq:dZk}, with $\|\Delta\mathbf{z}_{k'}\|_{\mathbf{Q}^{-1}} \leq\underline{\alpha}$, for all $k', 1\leq k'\leq k $. Then, 
\begin{equation}
\sum_{i=k-\mathcal{T}+1}^{k}\|\Delta\mathbf{z}_i\|_{\mathbf{Q}^{-1}}^2 \leq \sum_{i=k-\mathcal{T}+1}^{k}\underline{\alpha}^2 
=\alpha_{\chi^2}^2(\varepsilon,\mathcal{T}p,h).
\label{eqn:c3}
\end{equation}
In addition, we define $c_{max}=\max( c_1,\dots,c_\mathcal{T} )$ and 
\begin{equation} 
{\underline{g_k^a}} = \sum_{i=k-\mathcal{T}+1}^{k}c_{max}{\mathbf{z}_i^a}^T\mathbf{Q}^{-1}\mathbf{z}_i^a,
\label{eqn:c4}
\end{equation}
as well as 
$\underline{\beta_{k}^a}=P(\underline{g_k^a}>h)$.
From~\eqref{eqn:c4}, $\underline{g_k^a}$ is a scaled sum of noncentral $\chi^2$ distributions with $p$ degrees of freedom, so $\underline{g_k^a}/c_{max}$ will have the noncentral $\chi^2$ distribution with $p\mathcal{T}$ degrees of freedom and the  central moment equal to 
\begin{equation}
\underline{\lambda} = \sum_{i=k-\mathcal{T}+1}^{k}\|\Delta\mathbf{z}_i\|_{\mathbf{Q}^{-1}}^2.
\label{eqn:c6}
\end{equation}

Since $\underline{\beta_{k}^a}=P(\underline{g_k^a}>h)=P(\underline{g_k^a}/c_{max}>h/c_{max})$, following the proof of Theorem~\ref{thm:equConds} for the $\underline{g_k^a}/c_{max}$ detection function we have that $\underline{\beta_{k}^a}\leq\beta_{k}+\varepsilon$ is satisfied if and only if $\sqrt{\underline{\lambda}}\leq\alpha_{\chi^2}(\varepsilon,p\mathcal{T},h/c_{max})$. That is, using~\eqref{eqn:c6}
\begin{equation}
\left( \underline{\beta_{k}^a}\leq\beta_{k}+\varepsilon\right) \Leftrightarrow \sum_{i=k-\mathcal{T}+1}^{k}\|\Delta\mathbf{z}_i\|_{\mathbf{Q}^{-1}}^2 \leq \alpha_{\chi^2}^2(\varepsilon,p\mathcal{T},h/c_{max}). 
\label{eqn:c8}
\end{equation}
Since~\eqref{eqn:c3} follows from the condition of the theorem, from~\eqref{eqn:c8} we have that $\underline{\beta_k^a}\leq\beta_k+\varepsilon$ is satisfied. From~\eqref{eqn:c4} we have that $ g_k^a\leq \underline{g_k^a}$, meaning that $\beta_k^a \leq \underline{\beta_k^a}$. Thus, $\left(\beta_k^a\leq\beta_k+\varepsilon\right)$ holds 
, and $\mathbf{a}_{1..k}\in {\mathcal{A}}_k$ (i.e.,~$\underline{\mathcal{A}}_k\subseteq\mathcal{A}_k$). 
\end{proof}

\begin{lemma}
For a system with the detector function $g_k^a$ of the form from~\eqref{eq:AttackedGk}, the set of all stealthy attacks at time $k$, $\mathcal{A}_k$, can be overapproximated by the set
	\begin{equation}
		\overline{\mathcal{A}}_k=\{ \mathbf{a}_{1..k} | \|\Delta\mathbf{z}_{k'}\|_{\mathbf{Q}^{-1}}\leq\overline{\alpha},~\forall k', 1\leq k'\leq k \},
	\label{eq:OappA}
	\end{equation}
(i.e.,~$\overline{\mathcal{A}}_k\supseteq \mathcal{A}_k$), where $\overline{\alpha} = \alpha_{\chi^2}(\varepsilon,p,h/{c_\mathcal{T}})$.
\label{lem:overapp}
\end{lemma}
%
\begin{proof}
Consider an attack sequence  $\mathbf{a}_{1..k}\in {\mathcal{A}}_k$ with the detector function $g_k^a$ from~\eqref{eq:AttackedGk}.
Let $\overline{g_k^a}=c_{{\mathcal{T}}}{\mathbf{z}_k^a}^T\mathbf{Q}^{-1}\mathbf{z}_k^a$. Since $\overline{g_k^a} \leq g_k^a$ it follows that $\overline{\beta_k^a}=P(\overline{g_k^a}>h)\leq\beta_k^a$, where $\beta_k^a$ is defined as in~\eqref{eqn:beta_ka}.
Since $\mathbf{a}_{1..k}$ are stealthy, it follows that 
$\beta_k^a\leq\beta_k+\varepsilon$, and thus 
$\overline{\beta_k^a}\leq\beta_k+\varepsilon$ holds. 

On the other hand, the function $\overline{g_k^a}/c_{\mathcal{T}}$ has the $\chi^2$ distribution; by following the proof of Theorem~\ref{thm:equConds} for $\overline{g_k^a}/c_{\mathcal{T}}$ we have that $\overline{\beta_k^a}\leq\beta_k+\varepsilon$ is satisfied if and only if $\|\Delta\mathbf{z}_{k}\|_{\mathbf{Q}^{-1}}\leq\overline{\alpha} =  \alpha_{\chi^2}(\varepsilon,p,h/c_{\mathcal{T}})$. 
Therefore, we have that $\beta_{k'}^a\leq\beta_{k'}+\varepsilon$ implies $\|\Delta\mathbf{z}_{k'}\|_{\mathbf{Q}^{-1}}\leq\overline{\alpha},~k' =1, ..., k$, meaning that $\mathbf{a}_{1..k}\in\overline{\mathcal{A}}_k$ (i.e.,~$\mathcal{A}_k\subseteq \overline{\mathcal{A}}_k$). 
\end{proof}

\begin{remark}
	The previous lemmas also hold for the detection function
$
		g_k = \sum_{i=1}^k c_k\mathbf{z}_i^T\mathbf{Q}^{-1}\mathbf{z}_i;
$ 
this can be shown by replacing~$\mathcal{T}$ with $k$ in the previous analysis, since it would not affect their proofs. In essence, this means that these results hold for both windowed detectors and SPRT detectors -- SPRT detectors are explored in detail in Section~\ref{sec:framework}.
\end{remark}

Lemmas~\ref{lem:underapp} and~\ref{lem:overapp} introduce attack sets $\underline{\mathcal{A}}_k$ and $\overline{\mathcal{A}}_k$ for which the attack constraints are captured as robustness bounds on $\|\Delta\mathbf{z}_{k}\|_{\mathbf{Q}^{-1}}$ instead of probabilities of attack detection, and for which 
$\underline{\mathcal{A}}_k\subseteq \mathcal{A}_k\subseteq \overline{\mathcal{A}}_k$. Hence, to analyze impact of stealthy attacks, we can consider the effects of attacks that have to maintain  $\|\Delta\mathbf{z}_{k}\|$ below a certain threshold. 

\begin{theorem}
\label{theorem:equibound}
	$\mathcal{R}_k$ from~\eqref{eqn:Rk} is bounded if and only if the set 
	\begin{equation}
		\label{eqn:RkCap}
		{\hat{\mathcal{R}}_k^{\alpha}}=\left\{ \begin{array}{c|c} \Delta\mathbf{e}_k \in\mathbb{R}^n &  \left.\begin{array}{c} \Delta\mathbf{e}_k,\Delta\mathbf{z}_k~\hbox{satisfy \eqref{eq:dEk} and \eqref{eq:dZk}},\\~\Delta\mathbf{e}_{k-1}\in \hat{\mathcal{R}}_{k-1}^{\alpha}, 
\|\Delta\mathbf{z}_k\|_2\leq \alpha\end{array}\right.\end{array}\right\},
	\end{equation}
is bounded, where ${\hat{\mathcal{R}}_0^{\alpha}}={\mathbf{0}}\in\mathbb{R}^n$ and $\alpha>0$.
\end{theorem}

\begin{proof}
From~\eqref{eq:eVariability}, $\gamma Cov(\mathbf{e}_k^a)$ is bounded and 
we can simplify our presentation by focusing on the case where $\gamma=0$. 
Furthermore, for any vector $\mathbf{v}$, the set $\{\begin{array}{c|c} \mathbf{e} \in\mathbb{R}^n &  \mathbf{e}\mathbf{e}^T\preccurlyeq \mathbf{v}\mathbf{v}^T \end{array}\}$ 
is bounded if and only if the vector $\mathbf{v}$ is bounded. Therefore, the set $\left\{\begin{array}{c|c} \mathbf{e} \in\mathbb{R}^n &  \mathbf{e}\mathbf{e}^T\preccurlyeq E[\mathbf{e}_k^a]E[\mathbf{e}_k^a]^T+\gamma Cov(\mathbf{e}_k^a) \end{array}\right\}$ will be bounded if and only if $E[\mathbf{e}_k^a]=\Delta\mathbf{e}_k$ (from~\eqref{eq:ExpDeltaEq}) is bounded.
	
Consider attack vectors $\mathbf{a}_{1..k}\in\mathcal{A}_k$. From Lemmas~\ref{lem:underapp} and~\ref{lem:overapp} we have that
\begin{equation}
\label{eqn:c9}
\{\Delta\mathbf{e}_k|\mathbf{a}_{1..k}\in\underline{\mathcal{A}}_k\} \subseteq \{\Delta\mathbf{e}_k|\mathbf{a}_{1..k}\in\mathcal{A}_k\} \subseteq \{\Delta\mathbf{e}_k|\mathbf{a}_{1..k}\in\overline{\mathcal{A}}_k\},
\end{equation}
where we somewhat abuse the notation, by having $\{\Delta\mathbf{e}_k|\mathbf{a}_{1..k}\in\mathcal{A}\}$ capture all reachable vectors $\Delta\mathbf{e}_k$ when the system~\eqref{eq:dEk} is `driven' by attack vectors from the set~$\mathcal{A}$. On the other hand, from linearity of the system described by~\eqref{eq:dEk} and~\eqref{eq:dZk}, the sets $\{\Delta\mathbf{e}_k|\|\Delta\mathbf{z}_{k'}\|_{\mathbf{Q}^{-1}}
\leq\underline{\alpha}, k'=1,..., k\}$ and $\{\Delta\mathbf{e}_k|\|\Delta\mathbf{z}_{k'}\|_{\mathbf{Q}^{-1}} \leq\overline{\alpha},k'=1,..., k\}$  are either both bounded or both unbounded. Thus, from~\eqref{eqn:c9}, these sets are bounded if and only if $ \{\Delta\mathbf{e}_k|\mathbf{a}_{1..k}\in\mathcal{A}_k\}$ is bounded.

Finally, as  
$	\frac{1}{\abs{\lambda_{max}}}\|\Delta \mathbf{z}_{k}\|_2 \leq \|\Delta \mathbf{z}_{k}\|_{\mathbf{Q}^{-1}} \leq \frac{1}{\abs{\lambda_{min}}}\|\Delta \mathbf{z}_{k}\|_2,
	\label{eq:OverAppP-2}
	$
where $\lambda_{max}, \lambda_{min}$ are the largest and smallest, respectively, eigenvalue of $\mathbf{Q}$, the region $\hat{\mathcal{R}}^{\alpha}_k$ will be bounded for the constraint $\|\Delta\mathbf{z}_k\|_{\mathbf{Q}^{-1}} \leq \alpha$ if and only if its bounded with a 2-norm stealthiness constraint
$
\|\Delta \mathbf{z}_k\|_2\leq \alpha
$ from~\eqref{eqn:RkCap}.
\end{proof}

\subsection{Perfectly Attackable Systems}

Theorem~\ref{theorem:equibound} can be used to formally capture dynamical systems for which there exists a stealthy attack sequence that results in an unbounded state estimation error -- i.e., for such systems, given enough time, the attacker can make arbitrary changes in the system states without risking detection. 

\begin{definition}
A system is \emph{perfectly attackable} (PA) if the system's reachable set $\mathcal{R}$ from~\eqref{eqn:R} is an unbounded set.
\end{definition}

As shown in~\cite{mo2010falseincsys,kwon2014stealthy}, for LTI systems without any additional data integrity guarantees, 
the set ${\hat{\mathcal{R}}^{\alpha}}=\bigcup_{k=0}^\infty \hat{\mathcal{R}}_k^{\alpha}$ 
can be bounded or unbounded depending on the system dynamics and the set of compromised sensors $\mathcal{K}$. 
From Theorem \ref{theorem:equibound}, this property is preserved for the set $\mathcal{R}$ as well. For this reason, we will be using the definition of $\hat{\mathcal{R}}^{\alpha}$ to analyze boundedness of $\mathcal{R}$, and to simplify the notation due to linearity of the constraint we will assume that $\alpha =1$ -- i.e.,~for this analysis we consider the stealthiness attack constraint as
\begin{equation}
\label{eqn:z_th}
\|\Delta \mathbf{z}_k\|_{2} \leq 1, \quad~\quad k\in\mathbb{N}_0, 
\end{equation}
{imposed on the system $\Xi$ from~\eqref{eq:dEk} and~\eqref{eq:dZk}.}


Now, the theorem below follows from~\cite{mo2010falseincsys,kwon2014stealthy}. 



\begin{theorem}
\label{thm:PA}
A system from~\eqref{eq:BadSens} is perfectly attackable if and only if the matrix $\mathbf{A}$ is unstable, and at least one eigenvector $\mathbf{v}$ corresponding to an unstable 
 eigenvalue satisfies $\hbox{supp}(\mathbf{Cv}) \subseteq \mathcal{K}$ and $\mathbf{v}$ is a reachable state of the system $\Xi$ from~$\eqref{eq:dEk},~\eqref{eq:dZk}$.
\end{theorem}

{Note that~\cite{mo2010falseincsys} also uses the term \emph{unstable eigenvalue} $\lambda$ to denote $|\lambda|\geq 1$.}
In the next section, we show that intermittent integrity~guarantees significantly limit stealthy attacks even for \emph{perfectly-attackable}~systems.

\section{Stealthy Attacks in Systems with Intermittent Integrity Enforcement }
\label{sec:se_wIntegrity}


In this section, we analyze  the effects that intermittent data integrity guarantees have on the estimation error under attack.
To formalize this notion, we start with the following definition.


\begin{definition}
\label{def:policy}
A \emph{global} intermittent data integrity enforcement policy $(\mu, f, L)$, where $\mu=\left\{t_{k}\right\}_{k=0}^\infty$ such that  $t_{0}>1$, for all $k >0$, $t_{{k-1}}<t_{k}$ and $L=\sup_{k>0} \left(t_{k} - t_{{k-1}}\right)$, ensures that 
$$\mathbf{a}_{t_k} = \mathbf{a}_{t_k+1} =...=\mathbf{a}_{t_k+f-1} =\mathbf{0}, \forall k\geq 0.$$ 
Furthermore, for a sensor $s_i\in\mathcal{S}$, the sensor's intermittent data integrity enforcement policy $(\mu_i,f_i, L_i)$, where {$\mu_i=\left\{t^i_{k}\right\}_{k=0}^\infty$ with $t^i_{0}>1$,  $t^i_{k-1}<t^i_k$ for all $k>0$, and $L_i=\sup_{k>0} \left(t^i_{k} - t^i_{k-1}\right)$, ensures that 
$$\mathbf{a}_{t^i_{k},i} = \mathbf{a}_{t^i_{k}+1,i}=...=\mathbf{a}_{t^i_{k}+f_i-1,i}  =0,  \forall k\geq 0.$$ 
}
\end{definition} 


%

 Intuitively, an intermittent data integrity enforcement policy for sensor $s_i$ ensures that the injected attack $\mathbf{a}_{k,i}$ via the sensor  will be equal to zero in at least $f_i$ consecutive points, where the starts of these `blocks' are at most $L_i$ time-steps apart. 
Similarly, for a \emph{global} intermittent data integrity enforcement policy, the whole attack vector $\mathbf{a}_{k}$ has to be $\mathbf{0}$ for at least $f$ consecutive steps, and the duration between these blocks is bounded from above to at most $L$ time-steps.


Global intermittent integrity enforcement is easier to model (and analyze, as we will show in the next section). However, compared to the use of $p$ separate sensor's intermittent integrity enforcements, global enforcement policies impose significantly larger communication and computation overhead in every time-step when data integrity is enforced. For example, with global enforcement every sensor has to be able to compute and add a MAC to its message transmitted over a shared bus during one sampling period (which usually corresponds to a single communication frame).
In addition, since in these systems estimation/control updates are commonly computed once all messages are received, when the integrity is enforced 
the estimator has to be able to evaluate/recompute all received MACs before its execution for that time-period. 
On the other hand, with integrity enforcement for each sensor, their MACs can be sent and reevaluated in separate (e.g., consecutive) sampling periods (i.e.,~communication frames).




\begin{remark}
It is worth noting that our definition of intermittent integrity enforcement policies imposes a maximum time between integrity enforcements which, as we will show, is related to the worst-case estimation error caused by the attacks. The definition also captures \textbf{{periodic}} integrity enforcements when $L= t_{k} - t_{{k-1}}$ for all $k>0$. Finally, the definition also allows for capturing policies with continuous integrity enforcements, by specifying $L\leq f$.
\label{rem:PerPolicy}
\end{remark}


The following theorem specifies that when a global intermittent integrity enforcement policy is used a stealthy attacker cannot insert an unbounded expected state estimation error. 

\begin{theorem}
\label{thm:main_global}
Consider an LTI system 
from~\eqref{eq:ModlStSpc} 
with a global data integrity policy $(\mu, f, L)$, where 
{
\begin{equation}
\label{eqn:f}
f=\min(\psi,q_{un}),
\end{equation}
$L$ is finite, $\psi$ is the observability index of the $(\mathbf{A},\mathbf{C})$ pair, and $q_{un}$ denotes the number of 
 unstable eigenvalues of $\mathbf{A}$. 
 Then the system is not perfectly attackable. 
}
\end{theorem}

From the above theorem, it follows that even intermittent integrity guarantees significantly limit the damage that the attacker could make to the system. Furthermore, 
the theorem makes no assumptions about the set $\mathcal{K}$ of compromised sensors; 
in the general case, system designers may not be able to provide this type of guarantees during system design, and thus no restrictions are imposed on the set, neither regarding the number of elements or whether some sensors belong to~it. 

{
\begin{remark}
In our preliminary results reported in~\cite{jovanov_cdc17}, a similar formulation of Theorem~\ref{thm:main_global} is used with $f=\min \left(nullity(\mathbf{C})+1,q_{un}\right)$. Since 
 $\psi\leq n-rank(\mathbf{C})+1$ from~\cite{shreyas_notes}, using the rank--nullity theorem it follows that $\psi\leq nullity(\mathbf{C})+1$, meaning that the condition from Theorem~\ref{thm:main_global} is stronger than our earlier result and may further reduce the number of integrity-enforcement points.
\end{remark}
}


In the rest of the paper, we use the notation from Theorem \ref{thm:main_global} 
for $f$ and~$q_{un}$.
To show the theorem, we {exploit the following Lemma~\ref{lemma:bound} and Theorem~\ref{theorem:bound};   
 the lemma states that if stealthy attacks introduce unbounded estimation error $\Delta \mathbf{e}_k$, the unbounded components must belong to vector subspaces corresponding to unstable modes of the system (i.e.,~matrix~$\mathbf{A}$).
 }
 
{
\begin{lemma}
\label{lemma:bound}
Consider  system $\Xi$ from~\eqref{eq:dEk} and~\eqref{eq:dZk} under the stealthiness contraint~\eqref{eqn:z_th}, and let us denote by $\mathbf{v}_1,\dots,\mathbf{v}_{q_{un}}$ eigenvectors and generalized eigenvectors that correspond to unstable eigenvalues of matrix $\mathbf{A}$. 
Then, unbounded estimation errors $\Delta \mathbf{e}_k$ can be represented as
\begin{equation}
\label{eq:decomp}
	\Delta \mathbf{e}_k = \alpha_1\mathbf{v}_1 + \dots + \alpha_{q_{un}}\mathbf{v}_{q_{un}} + \mathbf{\varrho}_k
\end{equation}
where $\mathbf{\varrho}_k = \sum_{j=q_{un}+1}^{n} \alpha_j\mathbf{v}_j$ is a bounded vector, and
 for some  $1\leq i \leq q_{un}$ it holds that $\alpha_i\rightarrow\infty$ as $k\rightarrow\infty$.
\end{lemma}

\begin{proof} The proof is provided in the appendix. \end{proof}
}

\begin{theorem}
\label{theorem:bound}
Consider any $k\in\mathbf{N}$, such that $k+1\in\mu$ (i.e.,~at time $k+1$ an integrity enforcement block in the policy $\mu$ starts). If 
$\Delta \mathbf{e}_k$ is reachable state of  $\Xi$,  
and if~vectors $\mathbf{CA}\Delta\mathbf{e}_{k},\mathbf{CA}\Delta\mathbf{e}_{k+1},\dots,\mathbf{CA}\Delta\mathbf{e}_{k+f-1}$ are bounded, 
then the vector $\Delta\mathbf{e}_{k+f}$ has to be bounded for any stealthy~attack.\footnote{Formally, the theorem states that the subsequence $\{\Delta\mathbf{e}_{k+f}\}_{(k+1)\in\mu}$ of the sequence $\{\Delta\mathbf{e}_{k}\}_{k\in\mathbb{N}}$
 is bounded, if the subsequence $\{\mathbf{CA}\Delta\mathbf{e}_{k},\mathbf{CA}\Delta\mathbf{e}_{k+1}, \dots, \mathbf{CA}\Delta\mathbf{e}_{k+f-1}\}_{(k+1)\in\mu}$ of the sequence $\{\mathbf{CA}\Delta\mathbf{e}_{k}\}_{k\in\mathbb{N}}$ is bounded. However, to simplify our presentation and notation, we simply refer to the vectors, instead of subsequences, as~bounded.} 
\end{theorem}



\begin{proof}

From~\eqref{eq:dEk} and~\eqref{eq:dZk} it follows that
	\begin{align}
		\label{eq:dEkLemma}
		\Delta \mathbf{e}_{k+f} &= \mathbf{A}\Delta \mathbf{e}_{k+f-1} - \mathbf{K}\Delta \mathbf{z}_{k+f}, \\
		\label{eq:dZkLemma}
		\Delta \mathbf{z}_{k+f} &= \mathbf{CA}\Delta \mathbf{e}_{k+f-1}+\mathbf{ a}_{k+f} .
	\end{align} 
Since $\|\mathbf{K}\|_2$ is bounded, and $\|\Delta \mathbf{z}_{k+f}\|_2\leq 1$ due to the stealthy attack constraint~\eqref{eqn:z_th}, then $\|\mathbf{K}\Delta \mathbf{z}_{k+f}\|_2 \leq\|\mathbf{K}\|_2\|\Delta \mathbf{z}_{k+f}\|_2$ is bounded. Thus, to show that $\|\Delta \mathbf{e}_{k+f}\|_2$ is bounded, it is sufficient to prove that $\|\mathbf{A}\Delta \mathbf{e}_{k+f-1}\|_2$ is bounded.

Let's assume the opposite -- i.e.,~that $\|\mathbf{A}\Delta \mathbf{e}_{k+f-1}\|_2$ is unbounded while $\|\mathbf{CA}\Delta\mathbf{e}_{k}\|_2,\dots,\|\mathbf{CA}\Delta\mathbf{e}_{k+f-1}\|_2$ are all bounded. From~\eqref{eq:dEkLemma} it follows that 
$$\mathbf{A}\Delta\mathbf{e}_{k+f-1} = \mathbf{A}^{f}\Delta\mathbf{e}_k-\sum_{j=1}^{f-1}\mathbf{A}^{f-j}\mathbf{K}\Delta\mathbf{z}_{k+j}.$$

\noindent Given that $\|\Delta\mathbf{z}_{k+1}\|_2, \dots, \|\Delta\mathbf{z}_{k+f-1}\|_2$ are bounded due to the stealthy attack requirements, in order for $\mathbf{A}\Delta\mathbf{e}_{k+f-1}$ to be unbounded, $\mathbf{A}^{f}\Delta\mathbf{e}_k$ has to be unbounded as well.

{
Since $\mathbf{CA}\Delta\mathbf{e}_{k+1}$ is bounded, this implies that $\mathbf{CA}(\mathbf{A}\Delta\mathbf{e}_k - \mathbf{K}\Delta\mathbf{z}_{k+1})$ has to be bounded too. However, as $\Delta\mathbf{z}_{k+1}$ has to be bounded due to the stealthiness condition, it follows that $\mathbf{CA}^2\Delta\mathbf{e}_k$ has to remain bounded. Similarly, we can show that this holds up to $\mathbf{CA}^f\Delta\mathbf{e}_k$, and thus the vector  $\mathbf{b}_k(f)$ defined as
\begin{equation}
	 \mathbf{b}_k(f) = \underbrace{
	\begin{bmatrix}
		\mathbf{C} \\
		\mathbf{CA}\\
		\dots\\
		\mathbf{CA}^{f-1}
	\end{bmatrix}
	}_{\mathcal{O}_f}
	\mathbf{A}\Delta\mathbf{e}_k
\label{eq:ObsIndexCond}
\end{equation}
is bounded. Now, we consider two cases. 

\vspace{4pt}\noindent\textbf{Case I}: If $f$ is observability index of $(\mathbf{A},\mathbf{C})$ pair (i.e.,~$f=\psi$), then $\mathcal{O}_f$ has full rank, from which it follows that $\mathbf{A}\Delta\mathbf{e}_k$ (and thus $\mathbf{A}^{f}\Delta\mathbf{e}_k$) has to be also bounded, which is a contradiction. 

\vspace{4pt}\noindent\textbf{Case II}: Consider $f=q_{un}$, and let us use similarity transform $\mathbf{V}$ on the initial system, where $\mathbf{V}$ is defined as in the Lemma~\ref{lemma:bound}  proof -- i.e., $\mathbf{V}=[\mathbf{v}_1~...~\mathbf{v}_n]$ and we index (generalized) eigenvectors such that for each eigenvector $\mathbf{v}_{i}$ with $L_i$ generalized eigenvectors, $\mathbf{v}_{i+1},...,\mathbf{v}_{i+L_i}$ is its generalized eigenvector chain; 
in addition, $\mathbf{v}_1,..., \mathbf{v}_{q_{un}}$ are the eigenvectors (including generalized eigenvectors) for all unstable modes of~$\mathbf{A}$. 
 %

Thus, the transformed system can be captured as
\begin{equation*}
\begin{split}
 \tilde{\mathbf{A}} &= \mathbf{V}^{-1}\mathbf{A}\mathbf{V} = \mathbf{J} = \begin{bmatrix} {\mathbf{J}_1}_{(q_{un} \times q_{un})} & \mathbf{0}_{(q_{un} \times n-q_{un})} \\ \mathbf{0}_{(n-q_{un} \times q_{un})} & {\mathbf{J}_2}_{(n-q_{un} \times n-q_{un})} \end{bmatrix}
\\
\tilde{\mathbf{C}} &= \mathbf{CV} = \begin{bmatrix} \tilde{\mathbf{C}_1}_{(p\times q_{un})} & \tilde{\mathbf{C}_2}_{(p\times n-q_{un})} \end{bmatrix}
\label{eq:Ctil}
\end{split}
\end{equation*}
where $\mathbf{J}$ is the Jordan form of $\mathbf{A}$, 
$\mathbf{J}_1$ captures unstable modes of $\mathbf{A}$ and 
the pair $(\tilde{\mathbf{A}},\tilde{\mathbf{C}})$ is also  observable. 

Since $\mathbf{A}^{f}\Delta\mathbf{e}_k$ is unbounded we have that $\Delta\mathbf{e}_k$ is unbounded (from $\|\mathbf{A}^{f}\Delta\mathbf{e}_k\|_2\leq \|\mathbf{A}\|_2^{f}\|\Delta\mathbf{e}_k\|_2$). Thus, from Lemma~\ref{lemma:bound}, 
\begin{equation}
\Delta\mathbf{e}_k = \mathbf{V}\left[\alpha_1 \dots \alpha_n\right]^T = \mathbf{V}\alpha_{1..n},
\label{eq:de_proof}
\end{equation}
where 
$\alpha_{1..q_{un}}=\left[\alpha_1 \dots \alpha_{q_{un}}\right]^T$ is unbounded while 
$\alpha_{(q_{un}+1)..n} = \left[\alpha_{(q_{un}+1)} \dots \alpha_{q_{n}}\right]^T$ is a bounded vector. Due to the fact that 
$\tilde{\mathbf{C}}\tilde{\mathbf{A}}^j = \tilde{\mathbf{C}}{\mathbf{J}}^j  = \begin{bmatrix} \tilde{\mathbf{C}_1}\mathbf{J}_1^j & \tilde{\mathbf{C}_2}\mathbf{J}_2^j \end{bmatrix}$, from~\eqref{eq:ObsIndexCond} it follows that
\begin{equation}
\begin{split}
 \mathbf{b}_k(f) &= \begin{bmatrix}
\tilde{\mathbf{C}_1} & \tilde{\mathbf{C}_2} \\
\tilde{\mathbf{C}_1}\mathbf{J}_1 & \tilde{\mathbf{C}_2}\mathbf{J}_2 \\
\dots \\
\tilde{\mathbf{C}_1}\mathbf{J}_1^{f-1} & \tilde{\mathbf{C}_2}\mathbf{J}_2^{f-1}
\end{bmatrix}\mathbf{J}\alpha_{1..n} = \\
&=
\underbrace{
\begin{bmatrix}
\tilde{\mathbf{C}_1} \\
\tilde{\mathbf{C}_1}\mathbf{J}_1 \\
\dots \\
\tilde{\mathbf{C}_1}\mathbf{J}_1^{f-1}
\end{bmatrix}
}_{\tilde{\mathcal{O}}_{uns,f}} \mathbf{J}_1\alpha_{1..q_{un}}
+
\underbrace{
\begin{bmatrix}
\tilde{\mathbf{C}_2} \\
\tilde{\mathbf{C}_2}\mathbf{J}_2 \\
\dots \\
\tilde{\mathbf{C}_2}\mathbf{J}_2^{f-1}
\end{bmatrix}
}_{\tilde{\mathcal{O}}_{sta,f}} \mathbf{J}_2\alpha_{(q_{un}+1)..n}. 
\label{eq:ObsvSplit1}
\end{split}
\end{equation}
Since $\mathbf{b}_k(f)$ and $\alpha_{(q_{un}+1)..n}$ are bounded, from~\eqref{eq:ObsvSplit1} the vector 
 \begin{equation}
	\tilde{ \mathbf{b}}_k(f) = \tilde{\mathcal{O}}_{uns,f}\mathbf{J}_1\alpha_{1..q_{un}} 
	\label{eq:last_proof}
\end{equation}
is also bounded. 
Note that $\tilde{\mathcal{O}}_{uns,f}$ is effectively the observability matrix of the $(\mathbf{J}_1,\tilde{\mathbf{C}_1})$ pair corresponding to the subsystem with the $q_{un}$ unstable eigenvalues of $\mathbf{A}$. 

To show that $(\mathbf{J}_1,\tilde{\mathbf{C}_1})$ is observable, let us assume the opposite; thus, 
 there exist an eigenvector  $\tilde{\mathbf{v}}_j$ of $\mathbf{J}_1$ such that $\tilde{\mathbf{v}}_j\in null(\tilde{\mathbf{C}}_1)\subseteq\mathbb{R}^{q_{un}}$. Take note that $\mathbf{J}_1$ is a Jordan matrix, so each of its eigenvectors has to be  a projection vector $\mathbf{i}_j$ (as defined in Sec.~\ref{sec:notation}), where $j$, $1\leq j\leq q_{un}$, corresponds to the start of a Jordan block of $\mathbf{J}_1$. 
Yet this implies that $\tilde{\mathbf{C}}_1\mathbf{i}_j=\mathbf{0}_{p\times 1}$ -- i.e.,~the $j^{th}$ column of $\tilde{\mathbf{C}}_1$ and thus the $j^{th}$ column of $\tilde{\mathbf{C}}$ are zero vectors. However, since $ \tilde{\mathbf{C}} = \mathbf{CV} = \left[\mathbf{Cv}_1~\dots~\mathbf{Cv}_n \right]$, it follows that $\mathbf{Cv}_j = \mathbf{0}$ for some $j$. 
Due to the way $\mathbf{V}$ is formed and since $j$ has to be the start index of a Jordan block in $\mathbf{J}_1$, 
it follows that $\mathbf{v}_j$ is an eigenvector of $\mathbf{A}$. However, this implies that $\mathbf{v}_j\in null(\mathbf{C})$, making $(\mathbf{A},\mathbf{C})$ pair unobservable and contradicting our initial assumption about the system. 

Therefore, $(\mathbf{J}_1,\tilde{\mathbf{C}}_1)$ is observable meaning that 
$\tilde{\mathcal{O}}_{uns,f}$ is full rank. Furthermore, $\mathbf{J}_1$  is invertible as 
it contains only unstable (i.e., non-zero) eigenvalues of the system on the diagonal. Hence,~from \eqref{eq:last_proof} and the fact that $\mathbf{b}_k(f)$ is bounded it follows that 
 vector $\alpha_{1..q_{un}}$ has to be bounded, which from~\eqref{eq:de_proof} contradicts that  $\Delta\mathbf{e}_k$ and $\mathbf{A}^{f}\Delta\mathbf{e}_k$ are unbounded, and thus concludes the proof. 
}
\end{proof}

Using the previous theorem, we now prove Theorem~\ref{thm:main_global}.

\begin{proof}[Proof of Theorem~\ref{thm:main_global}]
Consider any time-point $t_k+f$ such that $t_k\in\mu$ -- i.e.,~$t_k$ is the start of an integrity enforcement block. Thus, 
$\mathbf{a}_{t_k} 
=...=\mathbf{a}_{t_k+f-1} =\mathbf{0}$. 
From~\eqref{eq:dZk} it follows that
$\Delta \mathbf{z}_{t_k+j} = \mathbf{CA}\Delta \mathbf{e}_{t_k+j-1}$, $j=0,...,f-1$, 
and thus from~\eqref{eqn:z_th} 
$$\|\mathbf{CA}\Delta \mathbf{e}_{t_k+j-1}\|_{2}\leq 1, \qquad j=0,...,f-1.$$

Now, from Theorem~\ref{theorem:bound} it follows that the state estimation error $\Delta\mathbf{e}_{t_k+f-1}$ has to be bounded for any stealthy attack; this holds for all time points at the ends of integrity enforcement intervals. Since in the proof of Theorem~\ref{theorem:bound}, we have not used any specifics of the time points, there exists a global bound on state estimation error at the end of all integrity enforcement periods (as illustrated in \figref{fig:bounds}). 

\begin{figure}[!t]%
\centering
\includegraphics[width=0.368\textwidth]{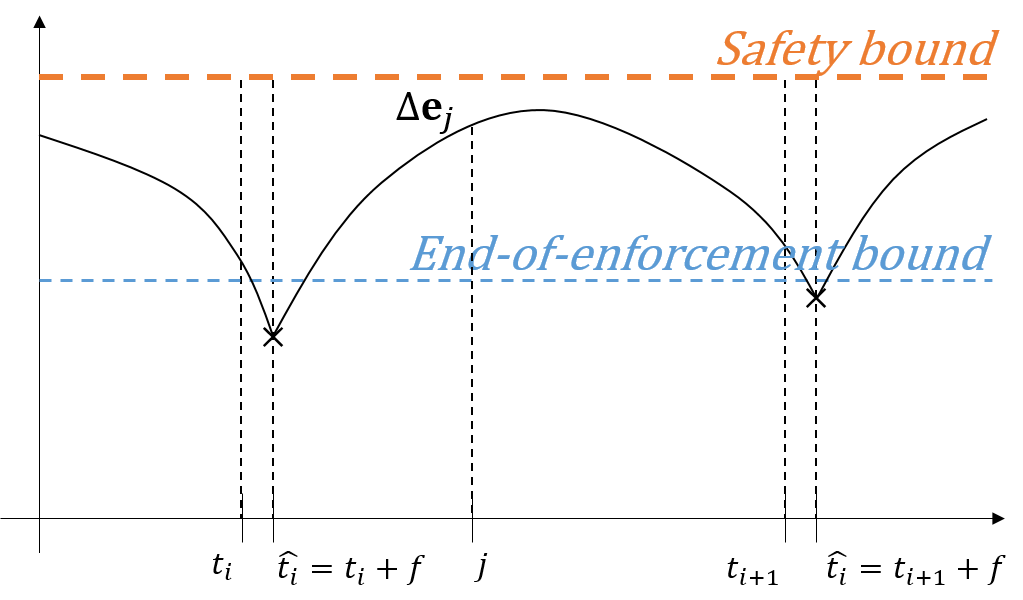}
\caption{System evolution between two consecutive endpoints of integrity enforcement intervals.}
\vspace{-6pt}
\label{fig:bounds}%
\end{figure}

Finally, consider an expected state-estimation error vector at any time~$j$. From Definition~\ref{def:policy}, there exists $t_i\in\mu$ such that $j\in\left[\hat{t}_i,\hat{t}_i+L\right)$, where $\hat{t}_i=t_i+f$ (\figref{fig:bounds}). 
Now, from~\eqref{eq:dEk} and~\eqref{eq:dZk} we have  that 
%
\vspace{-8pt}
\begin{equation}
\label{eqn:e1}
\Delta\mathbf{e}_{j} = \mathbf{A}^{j-\hat{t}_i}\Delta\mathbf{e}_{\hat{t}_i} -\sum_{l=1}^{j-\hat{t}_i}{\mathbf{A}^{j-\hat{t}_i-l}\mathbf{K}\Delta\mathbf{z}_{\hat{t}_i+l}}.
\end{equation}
Thus, the evolution of the expected state estimation error vector between two time points with bounded values can be described as evolution over a finite number of steps of a dynamical system with bounded inputs (since $\|\Delta\mathbf{z}_{\hat{t}_i+l}\|_2\leq 1$); 
from the triangle and Cauchy-Schwarz inequalities it follows
%
\begin{equation}
\label{eqn:bound_increase}
\|\Delta\mathbf{e}_{j}\|_2 \leq \|\mathbf{A}\|_2^{j-\hat{t}_i} \|\Delta\mathbf{e}_{\hat{t}_i}\|_2 + \sum_{l=1}^{j-\hat{t}_i}{\|\mathbf{A}\|_2^{j-\hat{t}_i-l}\|\mathbf{K}\|_2}.
\end{equation}
Hence, the expected estimation error vector $\Delta\mathbf{e}_{j} $ is bounded for any $j$, and the system is not perfectly attackable. 
\end{proof}

Theorem~\ref{thm:main_global} assumes that the attacker has the full knowledge of the system's integrity enforcement policy 
-- i.e.,~at which time-points integrity enforcements will occur. As we illustrate in Section~\ref{sec:case_study}, this allows the attacker to plan attacks that maximize the error, while ensuring stealthiness of the attack by reducing the state estimation errors to the levels that will not trigger detection during integrity enforcement intervals. 
On the other hand, if the attacker does not have the knowledge of $\mu$ {(i.e.,~if (s)he is not aware of the time points in which integrity enforcements would occur), the integrity enforcement requirements can be additionally relaxed; as the attacker does not know when enforcements occur, (s)he has to ensure that if at any future point (including the next time-step) malicious data cannot be injected, the residue would still remain below the threshold~\eqref{eqn:z_th}. Thus, we obtain the following result. 
}

\begin{theorem}
\label{thm:random}
Any LTI system from~\eqref{eq:ModlStSpc} with a global data integrity policy $(\mu, 1, L)$ ({i.e.,~with $f=1$}) is not perfectly attackable for any stealthy attacker {that does not know the time points $\mu$ when data integrity is enforced. }
\end{theorem}


\begin{proof}
{
First, note that the sequence $\mathbf{CA}\Delta\mathbf{e}_k$ cannot be bounded if the attacker wants to introduce an unbounded state estimation error. }
If it was bounded, 
from~\eqref{eq:dZk} and~\eqref{eqn:z_th} it would follow that $\mathbf{a}_k$ is always bounded; 
this in turn would imply that the system from~\eqref{eq:dEk} has bounded inputs, which since matrix $(\mathbf{A-KCA})$ is stable ($\mathbf{K}$ is Kalman gain) would imply that $\Delta\mathbf{e}_k$ cannot diverge -- the reachable set $\mathcal{R}$ can not be unbounded.


On the other hand, let us assume that the system is perfectly attackable -- i.e., the expected state estimation error can be unbounded. 
Then, {
from our previous argument it follows that $\mathbf{CA}\Delta\mathbf{e}_k$ is unbounded} and thus we can find $k$ and $\Delta\mathbf{e}_k$ such that $\|\mathbf{CA}\Delta\mathbf{e}_k\|_2>1$. Then, if global data integrity is enforced only once at the time-step $k+1$, from~\eqref{eq:dZk} it would follow that $\|\Delta \mathbf{z}_{k+1}\| = \|\mathbf{CA}\Delta\mathbf{e}_k\|_2>1$, which violates the stealthiness requirement from~\eqref{eqn:z_th}.
%
%
\end{proof}

Theorems~\ref{thm:main_global} and~\ref{theorem:bound} consider a worst-case scenario without any constraints or assumptions about the set of compromised sensors $\mathcal{K}$ (e.g., that less than $q$ sensors are compromised). Yet, some knowledge about the set $\mathcal{K}$ may be available at design-time. For instance, for \emph{MitM} attacks some sensors cannot be in set $\mathcal{K}$, such as on-board sensors that do not communicate over a network to deliver information to the estimator, 
or sensors with built-in continuous data authentication. In these cases, the number of integrity enforcements can be~reduced. 

\begin{corollary}
\label{cor:main_global}
Consider a system from~\eqref{eq:ModlStSpc} with a global data integrity policy $(\mu, f, L)$, where {$f=\min(\psi,q_{un}^*)$, $\psi$ is the observability index of $(\mathbf{A},\mathbf{C})$, and $q_{un}^*$ denotes the number of unstable eigenvalues $\lambda_i$ of $\mathbf{A}$} for which the corresponding eigenvector $\mathbf{v}_i$ satisfies  $\hbox{supp}\left(\mathbf{Cv}_i\right)\in\mathcal{K}$. Then the system is not perfectly attackable. 
%
\end{corollary}

\begin{proof}
The proof directly follows the proof of Theorem~\ref{thm:main_global}, with the only difference that all $\alpha_i\rightarrow\infty$ 
 {from Lemma~\ref{lemma:bound}} also have to correspond to the unstable eigenvectors 
$\mathbf{v}_i$ satisfying that $\hbox{supp}\left(\mathbf{Cv}_i\right)\in\mathcal{K}$; 
otherwise, consider $\alpha_i\rightarrow\infty$,
from a decomposition of a `large' $\Delta\mathbf{e}_k$ such that 
{$\|\mathbf{C}\alpha_i\mathbf{v}_i\|_2 \rightarrow\infty$ }
and 
 $\hbox{supp}\left(\mathbf{Cv}_i\right)\notin\mathcal{K}$. Then the components of the residue $\Delta \mathbf{z}_{k+1}$ whose indicies are in $\hbox{supp}\left(\mathbf{Cv}_i\right)$ but not in $\mathcal{K}$ (i.e.,~$P_{\hbox{supp}\left(\mathbf{Cv}_i\right)\setminus\mathcal{K}}\Delta \mathbf{z}_{k+1}$) cannot be influenced by the attack signal $\mathbf{a}_{k+1}$, meaning that their large values due to $\alpha_i\rightarrow\infty$ cannot be compensated for by the attack signal, and thus will violate the stealthiness condition~\eqref{eqn:z_th}.
%
\end{proof}

%
%
Let us recall Definition~\ref{def:attkSet} that introduced $\mathcal{A}_k$, the set of all stealthy attacks up to time $k$ --   
it only requires that attack vector $\mathbf{a}_{1..k}\in\mathcal{A}_k$ satisfies the stealthiness conditions up to time $k$. Thus, as shown in the proof of Theorem~\ref{thm:main_global}, the~attacker applying attack $\mathbf{a}_{1..k}\in\mathcal{A}_k$  may have to violate the stealthiness constraint during the next integrity enforcement block, since for those time-points $t$ when integrity is enforced  $\mathbf{a}_t=\mathbf{0}$. 
%
As the attacker's goal is to remain stealthy even during integrity enforcements, 
we consider policy-aware stealthy attack~sets. 
\begin{definition}
\label{def:attkSetEnf}
For an integrity enforcement policy $(\mu,f,L)$, the set of all policy-aware stealthy attacks up to time $k$ is 
%
\begin{equation*}
		\mathcal{{A}}_k^\mu = \left\{ \begin{array}{c|c} \mathbf{a}_{1..k} &  \left.\begin{array}{c}\mathbf{a}_{1..k'}\in\mathcal{A}_{k'},
\\k'=\min\left\{t ~|~ (t-f+1\in\mu) \wedge (t\geq k)\right\}
\end{array}\right.\end{array}\hspace{-14pt} \right\}.
	\label{eq:FeasibleAttacksEnf}
\end{equation*}
\end{definition}

Intuitively, the attacker will always plan attacks at least until the end of next integrity enforcement block (captured by~$k'$), while keeping the probability of detection as low. 
Thus, we also need to modify the definition of the k-reachable region~$\mathcal{R}_k$ (Def.~\ref{def:OriginalRegionR}), as it depends on the employed  set of stealthy~attacks.

\begin{definition}
\label{def:EnfRegionR}
The policy aware $k$-reachable region $\mathcal{R}^\mu_k$ of the state estimation error under the attack (i.e.,~$\mathbf{e}^a_k$) is the set 
%
\begin{equation}
\label{eqn:RkEnf}
\mathcal{{R}}_k^\mu=\left\{ \begin{array}{c|c} \mathbf{e} \in\mathbb{R}^n &  \left.\begin{array}{c} \mathbf{e}\mathbf{e}^T\preccurlyeq E[\mathbf{e}_k^a]E[\mathbf{e}_k^a]^T + \gamma Cov(\mathbf{e}_k^a),
\\~\mathbf{e}_k^a=\mathbf{e}_k^a(\mathbf{a}_{1..k}),~\mathbf{a}_{1..k}\in\mathcal{A}^\mu_k
\end{array}\right.\end{array}\right\}.
\end{equation} 
Furthermore, the
global policy-aware reachable region $\mathcal{R}^\mu$ of the state estimation error $\mathbf{e}^a_k$ is the set
%
\begin{equation}
\label{eqn:REnf}
\mathcal{R}^\mu = \bigcup_{k=0}^\infty \mathcal{R}^\mu_k.
\end{equation}
\noindent 
\end{definition}
The above definition introduces a region that can be reached by an attacker that both considers past behavior and plans accordingly into the future to avoid being detected. Since $\mathcal{A}^\mu_k \subseteq \mathcal{A}_k$, it directly follows that $\mathcal{R}^\mu\subseteq\mathcal{R}$, and the boundedness property holds. Finally, note that when no integrity enforcements are used it follows that $\mathcal{R}^\mu\equiv\mathcal{R}$.


\subsubsection{Guarantees with Sensor-wise Integrity Enforcement} 
\label{sec:se_wIntegrity_local}

Due to space constraint, we now consider the case where the system has one unstable eigenvalue $\lambda_1$ with the corresponding eigenvector $\mathbf{v_1}$, but the result can be generalized. 
Also, let's assume that all sensor integrity enforcement policies~use $f_i=1$ and have $t^i_k=t^{i+1}_k-1$ for all $k$ and all $i=1,...,p-1$ (i.e.,~sensors enforce integrity in consecutive points, first $s_1$, then $s_2$, etc); this also implies all $L_i$ are~equal. 

It can be shown that the system is not perfectly attackable in this case. The proof follows the ideas from the proofs of Theorems~\ref{thm:main_global} and~\ref{theorem:bound}. If $s_1$ integrity is enforced at $t^1_k=j$, that would mean that $\Delta\mathbf{z}_{j,1}=P_{\{s_1\}}\mathbf{CA}\Delta\mathbf{e}_{j-1}{= P_{\{s_1\}}\tilde{\mathbf{C}}\mathbf{J}\alpha_{1..n}}$ {as in Theorem~\ref{theorem:bound}}, and thus $\| P_{\{s_1\}}\tilde{\mathbf{C}}\mathbf{J}\alpha_{1..n}\|_2 \leq 1$. 
From Lemma~\ref{lemma:bound}, if $\Delta\mathbf{e}_{j-1}$ is unbounded, {only $\alpha_1\rightarrow\infty$, and thus $\mathbf{J}_1$ as in~\eqref{eq:ObsvSplit1} is scalar. To account for this, $P_{\{s_1\}}\mathbf{C}\mathbf{v}_1\lambda_1\alpha_{1}$ has to be zero, which} 
  implies that 
 {$\mathbf{v}_{1}\in{null}(P_{\{s_1\}}\mathbf{C})$}. 
 {Similarly, it can be shown that from $\Delta\mathbf{z}_{j,i}$, it follows that $\mathbf{v}_{1}\in{null}(P_{\{s_i\}}\mathbf{C})$ for $1\leq i \leq n$. }
This can be represented~as
$\mathbf{C}\mathbf{v}_1 = \mathbf{0}$, which since $\lambda_1\neq 0$ implies that $\mathbf{v}_1\in null(\mathbf{C})$. This is a contradiction because $(\mathbf{A,C})$ is observable from our initial assumptions. 



\section{Analysis and Design of Safe Integrity Enforcement~Policies} 
\label{sec:framework} 

In the previous section, we have shown that with even intermittent integrity enforcements a stealthy attacker cannot introduce an unbounded state estimation error, irrelevant of the set of compromised sensors $\mathcal{K}$. 
However, we still need to provide methods to evaluate whether a specific integrity enforcement policy ensures the desired estimation performance (i.e.,~state estimation error) even in the presence of attacks. 
Furthermore, our goal is to also provide a design framework to derive integrity enforcement policies that ensure 
that the state estimation errors remain within a desired region even under attack. 
Thus, in this section, we introduce a computationally efficient method to achieve this based on an efficient estimation of the reachable region $\mathcal{R}_k^{\mu}$ from~\eqref{eqn:RkEnf} for systems with intermittent data integrity enforcements. 

\subsection{Reachable State Estimation Errors with Intermittent Integrity Enforcements}
\label{sec:reachability}

Consider an LTI system from~\eqref{eq:ModlStSpc},~\eqref{eq:BadSens} with a global data integrity policy $(\mu, f, L)$.  
As in Definition~\ref{def:attkSetEnf}, we use $\mathbf{a}_{1..k} = [(\mathbf{a}_1)^T~...~(\mathbf{a}_k)^T]^T\in\mathbb{R}^{pk}$ to capture attack vectors up to step~$k$, where $supp(\mathbf{a}_j)=\tilde{\mathcal{K}}_j$, $j=1,...,k$, and
$$
\tilde{\mathcal{K}}_j = 
\left\{ \begin{array}{l} 
\emptyset, ~~~~ j-i\in \mu,~ \hbox{for some } i, ~0\leq i<f, \\
\mathcal{K},~~~~~~~~~~~~~\hbox{otherwise} \end{array}\right.
$$
Here, $\tilde{\mathcal{K}}_j$ captures the set of compromised sensor measurements received in step $j$ -- i.e.,~if data integrity is enforced at step $j$ then no measurements are compromised. 
In addition, let us define $supp(\mathbf{a}_{1..k})=\mathcal{Q}_k\subseteq \{1,...,pk\}$; note that $\mathcal{Q}_k$ effectively captures information about the applied integrity enforcement policy, and 
\begin{equation}
\label{eqn:sizeQk}
|\mathcal{Q}_k|=|\mathcal{Q}_{k-1}|+|\tilde{\mathcal{K}}_k|=\sum_{i=1}^k |\tilde{\mathcal{K}}_i|.
\end{equation}
From~\eqref{eq:dEk} and~\eqref{eq:dZk}, $\Delta \mathbf{e}_{k}$ and $\Delta \mathbf{z}_{k}$ can be captured in a non-recursive form as %
\begin{equation}
\begin{split}
	&\Delta \mathbf{e}_{k} = 
		-\underbrace{
		\left[ 
		\begin{array}{c|c|c}
			(\mathbf{A-KCA})^{k-1}\mathbf{K} & ... & \mathbf{K}
		\end{array} 
		\right]
		}_{\mathbf{M}_{k}}
		\mathbf{a}_{1..k}\\
	&\Delta \mathbf{z}_{k} = 
	\underbrace{
	\left[
	\begin{array}{c|c}
		-\mathbf{CAM}_{k-1} & \mathbf{I}
	\end{array}
	\right]
	}_{\mathbf{N}_{k}}
	\mathbf{a}_{1..k}\\
	%
\end{split}
\label{eq:FrameEquIneq}
\end{equation}
To incorporate the information about the sparsity of the attack vector, we use suitable projections onto $\mathcal{Q}_k$ and $\tilde{\mathcal{K}}_1$, ..., $\tilde{\mathcal{K}}_k$, which satisfy $\mathbf{P}_{\mathcal{Q}_k} = BlckDiag(\mathbf{P}_{\tilde{\mathcal{K}}_1}, ..., \mathbf{P}_{\tilde{\mathcal{K}}_k})$. In addition, it holds that $ \mathbf{P}_{\tilde{\mathcal{K}}_j}^\dagger =  \mathbf{P}_{\tilde{\mathcal{K}}_j}^T$, since $\mathbf{P}_{\tilde{\mathcal{K}}_j} \mathbf{P}_{\tilde{\mathcal{K}}_j}^T = I_{|\tilde{\mathcal{K}}_j|}$, for $j=1,...,k$, and thus $\mathbf{P}_{{\mathcal{Q}}_k}^\dagger =  \mathbf{P}_{{\mathcal{Q}}_k}^T$. 
Then,~\eqref{eq:FrameEquIneq} can be restated as
%
%
\begin{equation}
\begin{split}
	&\Delta \mathbf{e}_{k} = 
		-\underbrace{
		\left[ 
		\begin{array}{c|c|c}
			(\mathbf{A-KCA})^{k-1}\mathbf{K}\mathbf{P}_{\tilde{\mathcal{K}}_1}^{\dagger} & ... & \mathbf{K}\mathbf{P}_{\tilde{\mathcal{K}}_k}^{\dagger}
		\end{array} 
		\right]
		}_{\mathbf{M}_{k}\mathbf{P}_{\mathcal{Q}_k^{\dagger}}}
		\mathbf{P}_{\mathcal{Q}_k}\mathbf{a}_{1..k}\\
	&\Delta \mathbf{z}_{k} = 
	\underbrace{
	\left[
	\begin{array}{c|c}
		-\mathbf{CAM}_{k-1}\mathbf{P}_{\mathcal{Q}^{\dagger}_{k-1}} & \mathbf{P}_{\tilde{\mathcal{K}}_k}^{\dagger}
	\end{array}
	\right]
	}_{\mathbf{N}_{k}\mathbf{P}_{\mathcal{Q}_k^{\dagger}}}
	\mathbf{P}_{\mathcal{Q}_k}\mathbf{a}_{1..k}\\
	%
\end{split}
\label{eq:FrameEquIneqSparse}
\end{equation}
%
with matrices $\mathbf{M}_{k}\mathbf{P}_{\mathcal{Q}_k^{\dagger}}$ and $\mathbf{N}_{k}\mathbf{P}_{\mathcal{Q}_k^{\dagger}}$ capturing information about the time steps in which data integrity is enforced.

For the general form of the detection function $g_k$ it may not be possible to obtain a simple analytical solution for the regions $\mathcal{R}^\mu_k$ and $\mathcal{R}^\mu$. 
Therefore, in this section we will focus on a specific detection function employed by 
Sequential Probability Ratio Test (SPRT) detectors. However, the presented method can be extended  in similar fashion to cover other detectors, such as cumulative sum and generalized likelihood test. 
{SPRT observes two hypothesis, $\mathcal{H}_0:~\mathbf{z}_k\sim\mathcal{N}(0,\mathbf{Q})$ and $\mathcal{H}_1:\mathbf{z}_k\not\sim\mathcal{N}(0,\mathbf{Q})$. One issue that arises from using SPRT is its non-linearity, given that it accumulates the data until decision is reached, after which observation window is reset. In addition, the exact distribution for $\mathbf{z}_k$ under $\mathcal{H}_1$ is not known since the mean of compromised $\mathbf{z}_k$ (i.e.,~$\mathbf{z}_k^a$) changes over time, which causes issues with implementation of SPRT as it requires known distributions without time-varying parameters. To address the first issue, we assume that the attacker attempts to stay between decision thresholds, where the upper threshold is denoted by the previously introduced~$h$; 
the attacker never goes below lower decision threshold, i.e., $\mathcal{H}_0$ is never observed, as that would imply greater constraints on the attacker, effectively resulting in attacks that exert a lower estimation error.  Thus, under these assumptions, the stopping time of SPRT in a compromised system is arbitrarily~large.

To address the second challenge (i.e.,~unknown distribution for $\mathbf{z}^a_k$), 
we approximate the detection function by initializing log-likelihood ratio $\Lambda_k\equiv 0$ when the system is not under the attack, as previously proposed in e.g.,~\cite{kwon2015real,kwon_tac17}; 
this will ensure that $g_k$ does not go above the threshold without attack. 
Consequently, from these assumptions, it follows that the detector function of SPRT detector can be captured~as
}
%
\begin{equation}
\begin{split}
	g_k &= g_{k-1}+\Lambda_k =
	\sum_{\tau=1}^{k}(\frac{1}{2}\mathbf{z}_\tau^T\mathbf{Q}^{-1}\mathbf{z}_\tau + log~c\sqrt{(2\pi)^p det(\mathbf{Q})})=\\
	&=\frac{1}{2}\sum_{\tau=1}^{k}(\mathbf{z}_\tau^T\mathbf{Q}^{-1}\mathbf{z}_\tau) + k~log~c\sqrt{(2\pi)^p det(\mathbf{Q})}
\end{split}
\label{eq:SpecificGk}
\end{equation}
%
%
where $\Lambda_k=log\frac{f_a(\mathbf{z}_k)}{f(\mathbf{z}_k)}$
, $f_a$ and $f$ are probability density functions of the residuals under the attack and in regular operation respectively, and $c=e^{-\frac{p}{2}}/\sqrt{(2\pi)^p det(\mathbf{Q})}$ is a design constant initialized such that log-likelihood ratio $\Lambda_k\equiv 0$.
%
Thus, in this case the attacker's stealthiness constraint from~\eqref{eqn:beta_stealthy} (i.e.,~$P(g_k^a>h) \leq P(g_k>h) + \varepsilon $) can be captured as

\vspace{-10pt}
\footnotesize
$$ P\hspace{-2px}\left(\sum_{\tau=1}^{k}({\mathbf{z}_\tau^a}^T\mathbf{Q}^{-1}\mathbf{z}_\tau^a)\hspace{-2px}>\hspace{-1px}2h+kp\right)\hspace{-3px} \leq \hspace{-1px} \varepsilon +P\hspace{-2px}\left(\sum_{\tau=1}^{k}(\mathbf{z}_\tau^T\mathbf{Q}^{-1}\mathbf{z}_\tau)\hspace{-2px}>\hspace{-1px}2h+kp\right)$$
\normalsize
Given that these two sums have the non-central $\chi^2$ (left) and (central) $\chi^2$ distributions, from Theorem \ref{thm:equConds} and the proof of Lemma \ref{lem:underapp} it follows that the above constraint is equivalent to 
\begin{equation}
\label{eqn:temp1}
\sqrt{\sum_{\tau=1}^{k}\|\Delta\mathbf{z}_{\tau}\|^2_{\mathbf{Q}^{-1}}}\leq\alpha_{\chi^2}(\varepsilon,kp,2h+kp).
\end{equation}
%
On the other hand, from~{\eqref{eq:FrameEquIneqSparse} it follows that 
\begin{equation*}
\begin{split}
\sum_{\tau=1}^{k}\|\Delta\mathbf{z}_{\tau}\|_{\mathbf{Q}^{-1}}^2&=\sum_{\tau=1}^{k}\Delta\mathbf{z}_\tau^T\mathbf{Q}^{-1}\Delta\mathbf{z}_\tau\stackrel{}{=}\\
	\stackrel{}{=}\sum_{\tau=1}^{k} &{(\mathbf{P}_{\mathcal{Q}_k}\mathbf{a}_{1..k})}^T[\mathbf{N}_\tau \mathbf{P}_{\mathcal{Q}^{\dagger}_\tau}~~\mathbf{0}_{p\times (|\mathcal{Q}_k|-|\mathcal{Q}_\tau|)}]^T \mathbf{Q}^{-1} \\
	&~~~~~~~~~~[\mathbf{N}_\tau \mathbf{P}_{\mathcal{Q}^{\dagger}_\tau}~~\mathbf{0}_{p\times (|\mathcal{Q}_k|-|\mathcal{Q}_\tau|)}] \mathbf{P}_{\mathcal{Q}_k}\mathbf{a}_{1..k}.
\end{split}
\end{equation*}
Hence, from~\eqref{eqn:temp1}, the attacker's stealthiness constraint under considered integrity enforcement policy $\mu$ can be captured~as
\begin{equation}
	\| P_{\mathcal{Q}_k}\mathbf{a}_{1..k}\|_{\mathbf{\Theta}_k} \leq \alpha_{\chi^2}(\varepsilon,kp,2h+kp),
\label{eq:finalGk}
\end{equation}
where

\vspace{-12pt}
\footnotesize
\begin{equation}
	\mathbf{\Theta}_k=\sum_{\tau=1}^{k}[\mathbf{N}_\tau \mathbf{P}_{\mathcal{Q}_\tau}^{\dagger}~\mathbf{0}_{p\times (|\mathcal{Q}_k|-|\mathcal{Q}_\tau|)}]^T \mathbf{Q}^{-1} [\mathbf{N}_\tau \mathbf{P}_{\mathcal{Q}_\tau}^{\dagger}~\mathbf{0}_{p\times (|\mathcal{Q}_k|-|\mathcal{Q}_\tau|)}].
\label{eq:Theta}
\end{equation}
\normalsize


For the above matrix $\mathbf{\Theta}_k$, the following property holds.

\begin{lemma}
\label{lem:PosDefTheta}
For any $k\geq 1$, the matrix $\mathbf{\Theta}_k$  is positive definite. 
\end{lemma}
\begin{proof}
We start with the case when $k=1$. From Definition~\ref{def:policy}, data integrity is not enforced at $k=1$ and thus $ \mathcal{Q}_1=\tilde{\mathcal{K}}_1=\mathcal{K}$. 
Due to the way projection matrices are formed, we have that 
	$$\mathbf{P}_{\mathcal{Q}_1}^{\dagger T}\mathbf{P}_{\mathcal{Q}_1}^{\dagger} = \mathbf{I}_{|\mathcal{Q}_1|}\succ 0~~~~\hbox{and}~~~~\mathbf{\Theta}_1 = [\mathbf{P}_{\mathcal{Q}_1}^{\dagger}]^T \mathbf{Q}^{-1} [\mathbf{P}_{\mathcal{Q}_1}^{\dagger}].$$
	Since $\mathbf{Q}\succ 0$, it follows that $\mathbf{\Theta}_1\succ 0$ as well. 
	
	Now, consider the case $k\geq 2$ and let us assume that $\mathbf{\Theta}_{k-1}$ is positive definite. From~\eqref{eq:Theta} it follows that
\begin{equation}
	\mathbf{\Theta}_{k} \hspace{-1px}= 
	\underbrace{
	\hspace{-3px}\begin{bmatrix}
		\mathbf{\Theta}_{k-1} & \hspace{-3px}\mathbf{0}_{|\mathcal{Q}_{k-1}| \times |\tilde{\mathcal{K}}_k|} \\
		\mathbf{0}_{|\tilde{\mathcal{K}}_k| \times |\mathcal{Q}_{k-1}|} & \hspace{-3px}\mathbf{0}_{|\tilde{\mathcal{K}}_k| \times |\tilde{\mathcal{K}}_k|}
	\end{bmatrix}
	}_{\tilde{\mathbf{\Theta}}_{k-1}}
	\hspace{-3px}+\hspace{-1px}
	\underbrace{
		[\mathbf{N}_k P_{\mathcal{Q}_k}^{\dagger}]^T \mathbf{Q}^{-1} [\mathbf{N}_k P_{\mathcal{Q}_k}^{\dagger}]
	}_{\tilde{\mathbf{\Theta}}_{k}}
\label{eq:ThetSum}
\end{equation}
	and we consider the following two cases.
	
\vspace{4pt}
\noindent Case I: 
There does not exist $i$, such that $0\leq i<f$ and $k-i\in \mu$; this implies that integrity is not enforced at the step $k$ and $\tilde{\mathcal{K}}_k=\mathcal{K}$.
Because both ${\mathbf{\Theta}}_{k-1}\succ 0$ and $\mathbf{Q}\succ 0$,  both addends in~\eqref{eq:ThetSum} are positive semidefinite matrices, and $\mathbf{\Theta}_{k}\succeq 0$. 
In addition, 
 since $\mathbf{\Theta}_{k-1}$ is positive definite by assumption, $null\left(\tilde{\mathbf{\Theta}}_{k-1}\right) = \mathcal{R}\left([ \mathbf{0}_{  |\tilde{\mathcal{K}}_k| \times |\mathcal{Q}_{k-1}|} ~~\mathbf{I}_{|\tilde{\mathcal{K}}_k|} ]^T\right)$. 
Furthermore, from~\eqref{eq:FrameEquIneqSparse}, we have \eqref{eq:TetaLongMtx}.
%
\begin{figure*}
\begin{equation}
\begin{split}
	\tilde{\mathbf{\Theta}}_{k} = 
	\begin{bmatrix}
		(\mathbf{CAM}_{k-1}\mathbf{P}_{\mathcal{Q}_{k-1}}^{\dagger})^T \mathbf{Q}^{-1} \mathbf{CAM}_{k-1}\mathbf{P}_{\mathcal{Q}_{k-1}}^{\dagger} & (\mathbf{CAM}_{k-1}\mathbf{P}_{\mathcal{Q}_{k-1}}^{\dagger})^T \mathbf{Q}^{-1} \mathbf{P}_{\tilde{\mathcal{K}}_k}^{\dagger} \\
		{(\mathbf{P}_{\tilde{\mathcal{K}}_k}^{\dagger})}^T \mathbf{Q}^{-1} \mathbf{CAM}_{k-1}\mathbf{P}_{\mathcal{Q}_{k-1}}^{\dagger} & {(\mathbf{P}_{\tilde{\mathcal{K}}_k}^{\dagger})}^T \mathbf{Q}^{-1} \mathbf{P}_{\tilde{\mathcal{K}}_k}^{\dagger}
	\end{bmatrix}
\end{split}
\label{eq:TetaLongMtx}
\end{equation}
\end{figure*}
%
Given that ${(\mathbf{P}_{\tilde{\mathcal{K}}_k}^{\dagger})}^T \mathbf{Q}^{-1} \mathbf{P}_{\tilde{\mathcal{K}}_k}^{\dagger} \succ 0$, 
it follows that $null\left(\tilde{\mathbf{\Theta}}_{k}\right)$ cannot have non-zero vectors from 
 $\mathcal{R}\left([ \mathbf{0}_{  |\tilde{\mathcal{K}}_k| \times |\mathcal{Q}_{k-1}|} ~~\mathbf{I}_{|\tilde{\mathcal{K}}_k|} ]^T\right)$.~Therefore, 
 \begin{equation}
 \label{eqn:intersection}
 null(\tilde{\mathbf{\Theta}}_{k}) \cap null(\tilde{\mathbf{\Theta}}_{k-1}) = \{\mathbf{0}\}.
 \end{equation}
Now, assume that there exists a non-zero vector $\mathbf{v}$ such that $\mathbf{v}\in null({\mathbf{\Theta}}_{k})$ -- i.e.,~$(\tilde{\mathbf{\Theta}}_{k}+\tilde{\mathbf{\Theta}}_{k-1})\mathbf{v}=0$, and thus
$$ \mathbf{v}^T \tilde{\mathbf{\Theta}}_{k} \mathbf{v} = -\mathbf{v}^T \tilde{\mathbf{\Theta}}_{k-1} \mathbf{v}. $$ 
However, since $\mathbf{v}$ cannot be in the null-spaces of both matrices due to~\eqref{eqn:intersection}, and $\tilde{\mathbf{\Theta}}_{k-1}$ and $\tilde{\mathbf{\Theta}}_{k}$ are both positive semidefinite, this is a clear contradiction. 
Consequently, $\mathbf{\Theta}_k=\tilde{\mathbf{\Theta}}_{k}+\tilde{\mathbf{\Theta}}_{k-1}$ 
$null(\mathbf{\Theta}_k)=\{\mathbf{0}\}$, and since $\mathbf{\Theta}_k$ is a positive semidefinite matrix it holds that $\mathbf{\Theta}_k\succ 0$. 

\vspace{4pt}\noindent
Case II: There exists $i$, such that $0\leq i<f$ and $k-i\in \mu$; i.e.,~integrity is enforced at the step $k$. 
Thus, $|\tilde{\mathcal{K}}_k|=0$, so $\tilde{\mathbf{\Theta}}_{k-1}={\mathbf{\Theta}}_{k-1}$ is positive definite. 
Thus, since ${\tilde{\mathbf{\Theta}}_{k}}\succeq 0$, it follows that  $\mathbf{\Theta}_k=\tilde{\mathbf{\Theta}}_{k}+\tilde{\mathbf{\Theta}}_{k-1}$ is positive definite.
\end{proof}

Now, the specification of the stealthiness condition from~\eqref{eq:finalGk} allows us to obtain the following result. 
\begin{theorem}
	The $k$-reachable region $\mathcal{R}^\mu_k$ under a global data integrity enforcement policy $(\mu,f,L)$  can be represented~as
	\begin{equation}
	\begin{split}
	\mathcal{R}^\mu_k = \left\{ \mathbf{e}_k^a | \mathbf{e}_k^a{\mathbf{e}_k^a}^T \preccurlyeq \alpha_{\chi^2}^2 [\mathbf{M}_k P_{\mathcal{Q}_k}^{\dagger}~\mathbf{0}] 
		\mathbf{\Theta}_t^{-1}{[\mathbf{M}_k P_{\mathcal{Q}_k}^{\dagger}~\mathbf{0}]}^T + \gamma\mathbf{\Sigma} \right\} 
	\label{eq:GreachableTheorem}
	\end{split}
	\end{equation}
where $ \alpha_{\chi^2}^2=  \alpha_{\chi^2}^2(\varepsilon,tp,2h+tp)$, $t$ is the first end of an integrity enforcement block following $k$ -- i.e.,~the earliest time point such that
$t-f+1\in\mu$ and $k\leq t$, and $\mathbf{0}=\mathbf{0}_{|\mathcal{Q}_k| \times (|\mathcal{Q}_t|-|\mathcal{Q}_k|)}$.
\label{thm:CalculateR}
\end{theorem}
%
\begin{proof}
Consider the stealthiness constraints~\eqref{eq:finalGk} at time $t$, which can be written as
	\begin{equation}
		\alpha_{\chi^2}^2(\varepsilon,tp,2h+tp) - (\mathbf{P}_{\mathcal{Q}_t}\mathbf{a}_{1..t})^T \mathbf{\Theta}_t \mathbf{P}_{\mathcal{Q}_t}\mathbf{a}_{1..t} \geq 0.
	\label{eq:finalGkUpTotRef}
	\end{equation}	
	
	Now, using Schur complement and Lemma~\ref{lem:PosDefTheta}, we obtain
	\begin{equation}
		\begin{bmatrix}
			{\mathbf{\Theta}_t}^{-1} & \mathbf{P}_{\mathcal{Q}_t}\mathbf{a}_{1..t} \\
			{(\mathbf{P}_{\mathcal{Q}_t}\mathbf{a}_{1..t})}^T & \alpha_{\chi^2}^2(\varepsilon,tp,2h+tp)
		\end{bmatrix}
		\succcurlyeq
		0
	\label{eq:SchurTheta}
	\end{equation}
	%
{As the left hand side of~\eqref{eq:SchurTheta} is positive semidefinite, when multiplied by a matrix from the left, and its transpose from the right, this product will also be positive semidefinite. If we use the projection matrix $\mathbf{P}_{\{1,\dots,k, t+1\}}$ for this, 
we effectively reduce the matrix from~\eqref{eq:SchurTheta} by removing pairs of rows and columns corresponding to $\mathbf{a}_{k+1..t}$. Thus, we obtain that
}
	%
	\begin{equation}
		\begin{bmatrix}
			[\mathbf{I}_{|\mathcal{Q}_k|}~~\mathbf{0}]{\mathbf{\Theta}_t}^{-1}[\mathbf{I}_{|\mathcal{Q}_k|}~~\mathbf{0}]^T & \mathbf{P}_{\mathcal{Q}_k}\mathbf{a}_{1..k} \\
			{(\mathbf{P}_{\mathcal{Q}_k}\mathbf{a}_{1..k})}^T & \alpha_{\chi^2}^2(\varepsilon,tp,2h+tp)
		\end{bmatrix}
		\succcurlyeq
		0
	\label{eq:SchurThetaDeranked}
	\end{equation}
where 
$\mathbf{0}=\mathbf{0}_{|\mathcal{Q}_k| \times (|\mathcal{Q}_t|-|\mathcal{Q}_k|)}$.
Furthermore, 
 with condition~\eqref{eq:SchurThetaDeranked} we need to compute only single ${\mathbf{\Theta}_t}^{-1}$ for all points between integrity enforcement blocks, 
as constraints for prior attacks (i.e.,~time points before $t$) directly follow from~\eqref{eq:SchurThetaDeranked}.
	
		The LMI in \eqref{eq:LMIQR} follows from~\eqref{eq:SchurThetaDeranked} as it forms a quadratic representation. {We use this specific matrix as it allow us to argue about the stealthiness condition using $\Delta\mathbf{e}_k$ rather than $\mathbf{a}_{1..k}$.}
\begin{figure*}
	\begin{equation}
	\begin{split}
		&\begin{bmatrix}
			-\mathbf{M}_k \mathbf{P}_{\mathcal{Q}_k}^{\dagger} & \mathbf{0}_{n\times 1}\\
			\mathbf{0}_{1\times |\mathcal{Q}_k|} & 1
		\end{bmatrix}
		\begin{bmatrix}
			[\mathbf{I}~~\mathbf{0}]{\mathbf{\Theta}_t}^{-1}[\mathbf{I}~~\mathbf{0}]^T & \mathbf{P}_{\mathcal{Q}_k}\mathbf{a}_{1..k} \\
			{(\mathbf{P}_{\mathcal{Q}_k}\mathbf{a}_{1..k})}^T & \alpha_{\chi^2}^2(\varepsilon,tp,2h+tp)
		\end{bmatrix}
		\begin{bmatrix}
			-\mathbf{M}_k \mathbf{P}_{\mathcal{Q}_k}^{\dagger} & \mathbf{0}_{n\times 1}\\
			\mathbf{0}_{1\times |\mathcal{Q}_k|} & 1
		\end{bmatrix}^T
		\succcurlyeq
		0 \Longleftrightarrow\\
		&\Longleftrightarrow\begin{bmatrix}
			\mathbf{M}_k \mathbf{P}_{\mathcal{Q}_k}^{\dagger}[\mathbf{I}~~\mathbf{0}]{\mathbf{\Theta}_t}^{-1}[\mathbf{I}~~\mathbf{0}]^T{(\mathbf{M}_k \mathbf{P}_{\mathcal{Q}_k}^{\dagger})}^T & -\mathbf{M}_k\mathbf{P}_{\mathcal{Q}_k}^{\dagger} \mathbf{P}_{\mathcal{Q}_k}\mathbf{a}_{1..k} \\
			-{(\mathbf{P}_{\mathcal{Q}_k}\mathbf{a}_{1..k})}^T{(\mathbf{M}_k \mathbf{P}_{\mathcal{Q}_k}^{\dagger})}^T & \alpha_{\chi^2}^2(\varepsilon,tp,2h+tp)
		\end{bmatrix}
		\succcurlyeq
		0.
	\end{split}	
	\label{eq:LMIQR}
	\end{equation}
\end{figure*}
%
Using~\eqref{eq:FrameEquIneqSparse} and Schur complement once again, we have 
	\begin{equation}
		[\mathbf{M}_k P_{\mathcal{Q}_k}^{\dagger}~~\mathbf{0}]\mathbf{\Theta}_t^{-1}[{(\mathbf{M}_k P_{\mathcal{Q}_k}^{\dagger}~~\mathbf{0})}]^T 
		- \frac{1}{\alpha_{\chi^2}^2}\Delta\mathbf{e}_k{\Delta\mathbf{e}_k}^T \succcurlyeq 0,
	\label{eq:finalShapeSchur}
	\end{equation}
 where $\alpha_{\chi^2}^2=  \alpha_{\chi^2}^2(\varepsilon,tp,2h+tp)$. 
Hence, from~\eqref{eq:finalShapeSchur} and the definition of $\mathcal{R}_k^{\mu}$ from~\eqref{eqn:RkEnf} , 
as well as~\eqref{eq:ExpDeltaEq} and the fact that $Cov[\mathbf{e}_k^a]=\mathbf{\Sigma}$, we finally obtain
that~\eqref{eq:GreachableTheorem} holds.  
\end{proof}
%
%

The representation of the reachable set from~\eqref{eq:GreachableTheorem} can be simplified further. 
Let's define $\mathbf{Y}_k$ as
\begin{equation}
\label{eqn:Y}
\mathbf{Y}_k = \alpha_{\chi^2}^2(\varepsilon,tp,2h+tp) [\mathbf{M}_k P_{\mathcal{Q}_k}^{\dagger}~~~\mathbf{0}]\mathbf{\Theta}_t^{-1}{[\mathbf{M}_k P_{\mathcal{Q}_k}^{\dagger}~~~\mathbf{0}]}^T + \gamma\mathbf{\Sigma}.
\end{equation}
Then~\eqref{eq:GreachableTheorem} is equivalent to $\mathbf{Y}_k-\mathbf{e}_k^a{\mathbf{e}_k^a}^T \succcurlyeq 0$, and thus by using Schur complement 
we obtain an alternative representation of the k-reachable regions~as
\begin{equation}
	\mathcal{R}^\mu_k = \{ \mathbf{e}_k^a | {\mathbf{e}_k^a}^T\mathbf{Y}_k^{-1}\mathbf{e}_k^a \preccurlyeq 1 \} .
\label{eq:ImplementationReady}
\end{equation}
for the positive definite matrix $\mathbf{Y}_k$ defined in~\eqref{eqn:Y}. The above representation can be exploited for 
efficient computation of the reachable-regions.

Furthermore, as we described in Section~\ref{sec:attack_model}, the attacker's goal is to maximize the expected state estimation error $E[\mathbf{e}^a_{k}] = \Delta \mathbf{e}_k$. From the above discussion, 
the following corollary directly holds by considering the case when $\gamma=0$.

\begin{corollary}
\label{cor:Emax}
A any time $k$, the maximal norm of the expected state estimation error $\mathbf{e}_k^a$ caused by the attack satisfies 
\begin{equation}
	\max \|E[\mathbf{e}^a_{k}]\|_2 = \frac{1}{\sqrt{\lambda_{max}(\mathbf{\tilde{Y}}_k)}},
\label{eq:MaxEcomputation}
\end{equation}
where 
$\lambda_{max}(\mathbf{\tilde{Y}}_k)$ denotes the largest 
eigenvalue of the matrix~$\mathbf{\tilde{Y}}_k =  \alpha_{\chi^2}^2(\varepsilon,tp,2h+tp) [\mathbf{M}_k P_{\mathcal{Q}_k}^{\dagger}~~\mathbf{0}]\mathbf{\Theta}_t^{-1}{[\mathbf{M}_k P_{\mathcal{Q}_k}^{\dagger}~~\mathbf{0}]}^T$, 
and $t$ is the next end of integrity enforcement block -- i.e.,~the earliest time point such that
$t-f+1\in\mu$ and $k\leq t$. 
\end{corollary}

The above corollary provides a very efficient way to evaluate worst-case effects of attacks when an intermittent data integrity enforcement policy is used. By quantifying degradation of the expected state estimation error in the presence of attacks we can analyze the impact of the integrity enforcement policy on limiting the attacker, which can then be used for design of suitable integrity enforcement policies.

\subsection{Design of Periodic Integrity Enforcement Policies}
For policy design, it is necessary to be able to evaluate impact of an integrity enforcement policy $\mu$, not only on reachable regions $\mathcal{R}^\mu_k$, for any $k$, but even more importantly on $\mathcal{R}^\mu$ from~\eqref{eqn:REnf}. 
To achieve this, we have to obtain the terminating value $t$ from Theorem~\ref{thm:CalculateR}, or equivalently from~\eqref{eq:ImplementationReady}, such that the reachability analysis can be completed after $\mathcal{R}^\mu_t$ is obtained -- i.e.,~for which $\mathcal{R}^\mu=\mathcal{R}^\mu_{1..t}$, where $\mathcal{R}^\mu_{1..t}=\bigcup_{k=1}^t\mathcal{R}^\mu_k$. In the general case, the analysis may never terminate, depending on the particular policy $(\mu,f,L)$.  
Therefore, to simplify the analysis, in this section we focus on periodic integrity enforcement policies introduced in Remark~\ref{rem:PerPolicy}. 

For a periodic integrity enforcement policy $(\mu,f,L)$, consider $t_1$ and $t_2=t_1+L$ time points at which consecutive integrity enforcement blocks end -- i.e.,~$t_1-f+1\in\mu$~and $ t_2-f+1\in\mu$. 
From the proof of Theorem \ref{thm:CalculateR}, if the stealthiness requirements from the condition in~\eqref{eq:SchurTheta} are satisfied at any time $t\in\mu$, then they are satisfied for all $k<t$, since~\eqref{eq:SchurThetaDeranked} follows from~\eqref{eq:SchurTheta}. 
Given that $\mathbf{a}_{t_1-f+1}=...=\mathbf{a}_{t_1}=\mathbf{0}$ and $\mathbf{a}_{t_2-f+1}=...=\mathbf{a}_{t_2}=\mathbf{0}$, and that the stealthiness requirements remain consistent throughout the analysis, it follows that the evolution of the estimation error between two consecutive integrity enforcement blocks will depend only on $E[\mathbf{e}^a_{t_1}] = \Delta\mathbf{e}_{t_1}$ and $E[\mathbf{e}^a_{t_2}] = \Delta\mathbf{e}_{t_2}$, or more specifically $\mathcal{R}^\mu_{t_1}$ and $\mathcal{R}^\mu_{t_2}$. Thus, if  $ \mathcal{R}^\mu_{t_2} \subseteq \mathcal{R}^\mu_{t_1}$ and 
$\mathcal{R}^\mu_{1..t_2} \subseteq \mathcal{R}^\mu_{1..t_1} $, 
then no new estimation error values can be reached after time $t_2$ and the terminating time for the reachability analysis can be $t_1$, since after time $t_2$ as well as after all following ends of integrity enforcement blocks the state estimation errors would start from a subset of the error values from $\mathcal{R}_{t_1}^{\mu}$. 
%
%
In addition, when the above terminating condition is satisfied,  the global reachable region of the state estimation error can be obtained as
$ \mathcal{R}^\mu = \bigcup_{k=1}^{\infty}\mathcal{R}^\mu_k = \bigcup_{k=1}^{t_1}\mathcal{R}^\mu_k =\mathcal{R}^\mu_{1..t_1}$. 
%

Consequently, using Algorithm~\ref{alg:GeneratePolicy} we can compute a periodic integrity enforcement policy that maximizes $L$ (i.e., reduces the integrity enforcement rate) while limiting the attacker's influence. Specifically, the algorithm will result in the enforcement policy that ensures that the state of reachable estimation errors does not contain points outside the set of safe (i.e.,~acceptable) errors $\mathcal{R}_{\mathbf{e}^a}$. In our evaluations in the next section, we define  $\mathcal{R}_{\mathbf{e}^a}$ using a threshold  $\|\Delta\mathbf{e}_{max}\|_2$ for the maximal 2-norm of the expected state estimation error due to attacks. Thus, the safety condition in Line~18 of the algorithm is mapped into $max(\|\mathbf{e}_1^a\|_{max},\dots,\|\mathbf{e}_t^a\|_{max})\geq  \|\Delta\mathbf{e}_{max}\|_2$, where 
 $\|\mathbf{e}_k^a\|_{max} = \max \|E[\mathbf{e}^a_{k}]\|_2 $ as computed in~\eqref{eq:MaxEcomputation}}.
 
 Finally, while we do not provide any guarantees that  Algorithm~\ref{alg:GeneratePolicy} will always terminate, for all analyzed systems, including the case studies from the next section, the condition in Line~17 was always eventually satisfied. Therefore, for all considered systems we have been able to use the algorithm to obtain periodic integrity enforcement policies that ensure desired estimation performance even in the presence of attacks.

\begin{algorithm}[!t]
\caption{Procedure for design of periodic integrity enforcement policies.}
{\textbf{Inputs:} System model, 
safe reachable region $\mathcal{R}_{\mathbf{e}^a}$ for the state estimation error $\mathbf{e}^a$
}
\label{alg:GeneratePolicy}
\begin{algorithmic}[1]
%
\STATE Enforcement distance $L=0$
\REPEAT
	\STATE $L=L+1$
	\STATE Form policy $(\mu,f,L)$ such that distance between consecutive elements in $\mu$ is $L$ and $t_0=L$
	\STATE Assign $t=0$ and the reachable region $\mathcal{R}_{1..t}=\varnothing$
	\REPEAT 
		\STATE{$t_{old}=t$}
		\STATE{$\mathcal{R}_{1..t_{old}}=\mathcal{R}_{1..t}$}
		\STATE{$t=min\{t'|t'\in\mu~\wedge~t'>t_{old}\}$}
		\STATE{Compute $\mathbf{N}_{t_{old}+1},...,\mathbf{N}_t$, $\mathbf{M}_{t_{old}+1},...,\mathbf{M}_t$ from~\eqref{eq:FrameEquIneq}}
		\STATE{Compute $\mathbf{\Theta}_t$ from~\eqref{eq:Theta}}
		\STATE{Compute $\alpha(\varepsilon,tp,2h+tp)$}
		\FOR{$k=t_{old}+1,\dots,t$}
			\STATE{Compute $\mathcal{R}^\mu_k$ using~\eqref{eq:ImplementationReady}}
			\STATE{$\mathcal{R}_{1..t} = \mathcal{R}_{1..t}\cup\mathcal{R}^\mu_k$}
		\ENDFOR
	\UNTIL{$\mathcal{R}_{1..t}\subseteq\mathcal{R}_{1..t_{old}}$ and $\mathcal{R}_{t}\subseteq\mathcal{R}_{t_{old}}$}
\UNTIL{$\mathcal{R}_{1..t_{old}}\setminus\mathcal{R}_{\mathbf{e}^a} \neq \varnothing$ }
\STATE{Accept policy $(\mu,f,L-1)$}
\end{algorithmic}
\end{algorithm}

\section{Case Studies}
\label{sec:case_study}


\label{sec:case_studies}
In this section, on automotive case studies
 we illustrate how intermittent data integrity enforcements can ensure satisfiable control performance even in the presence of attacks. 
  {For both studies, sensor values are transmitted 
   over an internal vehicle's network, such as  commonly used CAN~bus. Note that in~\cite{lesi_rtss17}, we provide additional automotive case-studies (and the overall scheduling framework) for intermittent authentication of CAN-bus messages from system sensors, and in~\cite{lesi_tecs17}~we show~benefits of intermittent authentication on vehicle's ECU~scheduling.}

{
\color{red}
}

%

\subsection{Case Study: Vehicle Trajectory Following}
We start with 
 the model used in~\cite{kerns2014unmanned} to describe vulnerabilities and potential attacks on autonomous systems {adapted for two-axis tracking}; 
we obtain the following discretized models (with sampling period of $0.01s$) for each axis
\begin{equation}
	\mathbf{A}_d=\begin{bmatrix} 1 & 0.01 \\ 0 & 1 \end{bmatrix} \quad \mathbf{B}_d = \begin{bmatrix} 0.0001 \\ 0.01 \end{bmatrix} \quad \mathbf{C}_d=\begin{bmatrix} 1 & 0 \\ 0 & 1 \end{bmatrix}
\label{eq:UAVdisc}
\end{equation}
Assume that 
 the attacker can modify the values from all sensors
 . 
The system is perfectly attackable as the matrix $\mathbf{A}_d$ is unstable and $\hbox{supp}(\mathbf{Cv})\in \mathcal{K}$, since $\mathcal{K}=\mathcal{S}$.

%

\begin{figure}[!t]%
	\centering
	\includegraphics[width=0.48\textwidth]{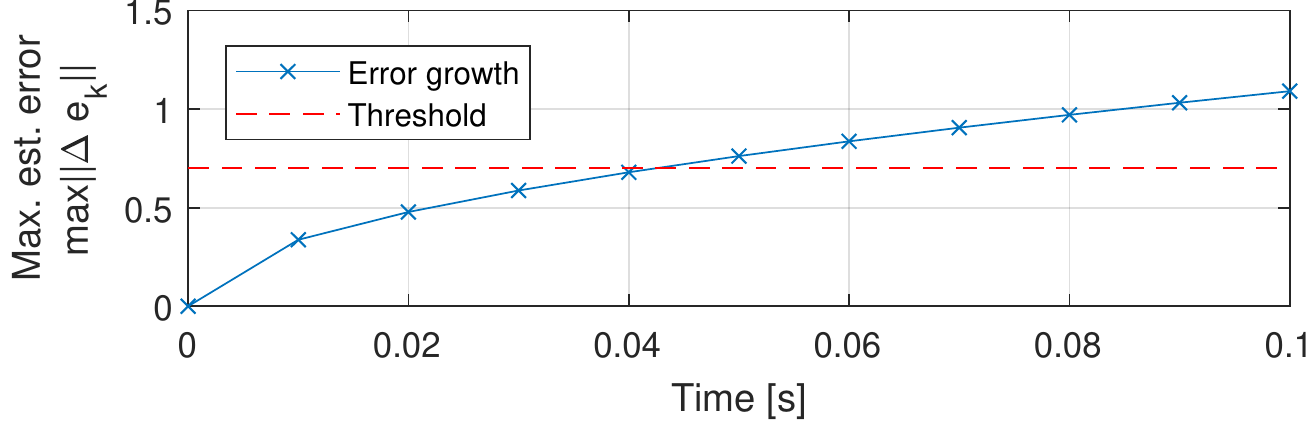}
	\caption{Evolution of the maximal estimation error for vehicle tracking; 
	without integrity enforcements, the attacker forces the system outside of the safe  range in 4 steps.}
	\label{fig:GAUAV}%
\end{figure}

 We consider the largest additive estimation error on position to be $0.5~m$ and on speed to be $0.5~\frac{m}{s}$, 
resulting in $\|\Delta\mathbf{e}_{max}\|_2=0.7$.
We also set 
such that
 the probability of false positive from~\eqref{eq:AlarmProb} 
  to $\beta=1.5\%$, and 
   additional probability of detection introduced by the attacker 
   from~\eqref{eqn:beta_stealthy} to $\varepsilon = 0.1\%$.

Without integrity enforcements, the attacker could force the state estimation error above  $\|\Delta\mathbf{e}_{max}\|_2$ threshold  after 4 steps, as shown in \figref{fig:GAUAV}. 
We 
 considered three periodic integrity enforcement policies with $f=1$ as specified in conditions of Theorem~\ref{thm:main_global}, and 
periods $L=20, 30$ and $35$, denoted by $\mu_{20}$, $\mu_{30}$, and $\mu_{35}$ respectively. Using 
 results from Section~\ref{sec:framework}, we show that the first two policies
  are safe, while the third policy 
 can violate the  $\|\Delta\mathbf{e}_{max}\|_2$ threshold --  \figref{fig:UAVTs} illustrates the evolution of the maximal estimation errors for each policy. 
 
\begin{figure}[!t]%
	\centering
	\includegraphics[width=0.48\textwidth]{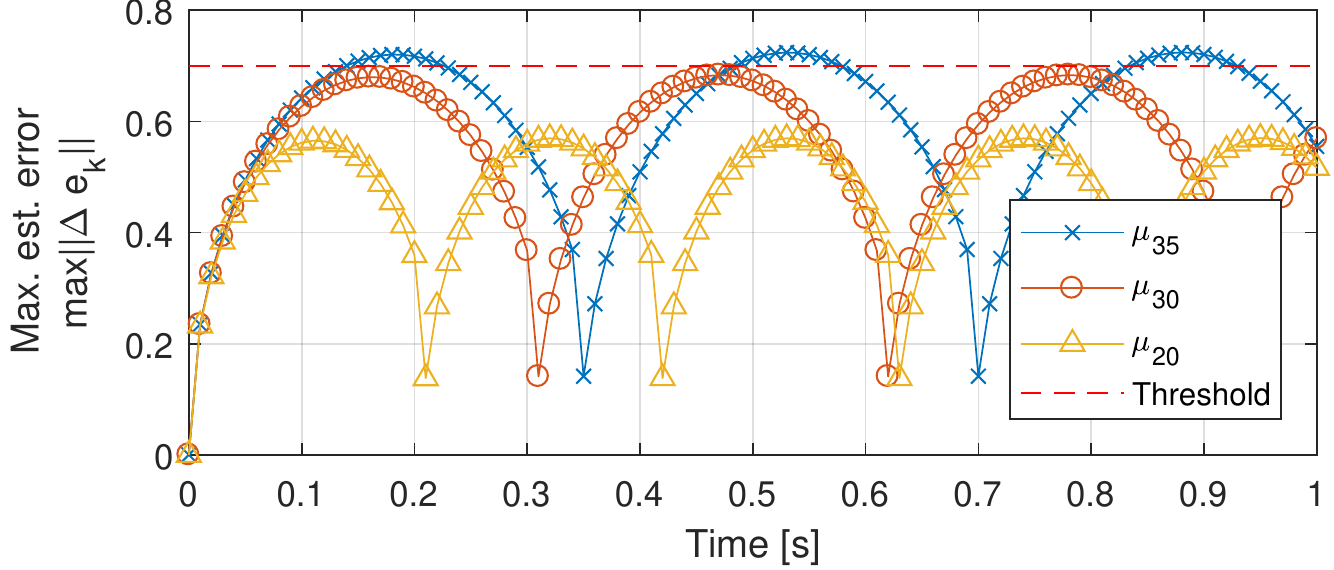}
	\caption{Maximal estimation error in the presence of attacks on all sensors for vehicle-tracking case study with three different integrity enforcement policies  with $f=1$ and periods $L=20, 30, 35$. 
	}
	\label{fig:UAVTs}%
\end{figure}

 
 Finally, we evaluat the effects of intermittent integrity~guarantees for trajectory following on a circular path with $100~m$ radius, at speed of $3.14~\frac{m}{s}$. \figref{fig:UAVattks} shows results of $200~s$ long simulations, with attacks starting at $100~s$. 
As illustrated,  when integrity is enforced on less than 3.4\% of messages,{~i.e., when $\mu_{30}$ is employed,}  we have strong control performance guarantees in the presence of attacks on all vehicle sensors.

%
\begin{figure*}[!t]
	\begin{center}
	\subfigure [State Estimates for system under stealthy attack, without integrity enforcement policies]
	{
		\includegraphics[width=0.312\textwidth]{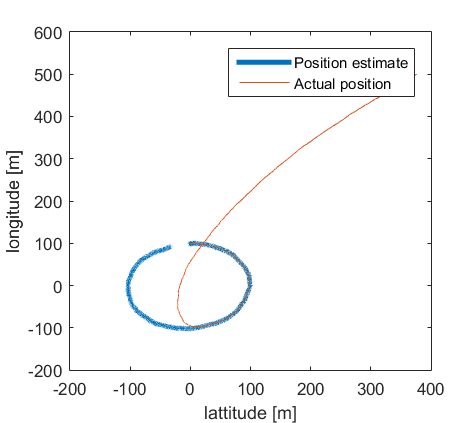}
		\label{fig:UAVflightGA}
	}
	\subfigure  [State estimates under stealthy attack with integrity enforcement policies.]
	{
		\includegraphics[width=0.312\textwidth]{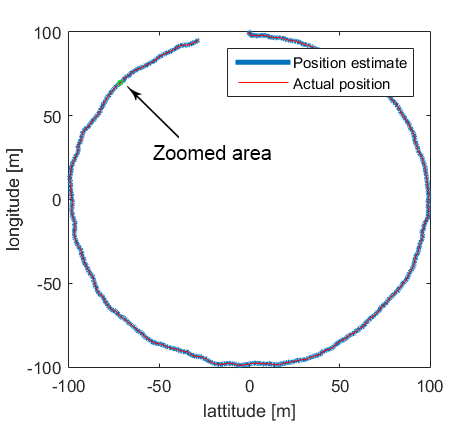}
		\label{fig:UAVflightEA}
	}
	\subfigure [Zoomed section of the \figref{fig:UAVflightEA}.]
	{
		\includegraphics[width=0.312\textwidth]{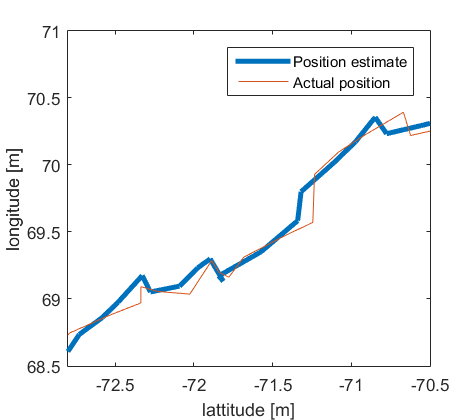}
		\label{fig:UAVflightEAzoom}
	}
	\end{center}
\vspace{-4pt}
\caption{State estimation of the tracked vehicle trajectory - without integrity enforcements a stealthy attacker can introduce a significant estimation error in a short period of time. However, even with intermittent integrity enforcement, the attack effects are negligible. Duration of the simulation is $200~s$ and the attack starts at $100~s$.}
\label{fig:UAVattks}
\end{figure*}

\subsection{ Degraded Cooperative Adaptive Cruise Control (dCACC)}
\label{subsec:CACC}

Cooperative Adaptive Cruise Control (CACC) employs 
 communication to obtain smaller following distance and better platooning stability than standard Adaptive Cruise Control. To achieve this, each vehicle is equipped with%
  a lidar and 
   acceleration {measurement sent} from the preceding vehicle. 
{However, when 
 acceleration data is not available 
 CACC needs to switch to dCACC, that is based only on local vehicle measurements. In this mode,
  Singer acceleration model 
  is used to estimate acceleration of the preceding vehicle~\cite{ploeg2015graceful} -- i.e.,}
\vspace{-2pt}

%
{
\begin{equation}
	\begin{bmatrix} \dot{d} \\ \dot{v} \\ \dot{a} \end{bmatrix} = \begin{bmatrix} 0 & -1 & 0 \\ 0 & 0 & 1 \\ 0 & 0 & -\frac{1}{\tau} \end{bmatrix} \begin{bmatrix} d \\ v \\ a \end{bmatrix} + \begin{bmatrix} 0 \\ 0 \\ 1 \end{bmatrix} \begin{bmatrix} u \end{bmatrix}
\label{eq:SingerAB}
\end{equation}
\vspace{-8pt}
\begin{equation}
	\mathbf{y} = \begin{bmatrix} 1 & 0 & 0 \\ 0 & 1 & 0 \end{bmatrix} \begin{bmatrix} d \\ v \\ a \end{bmatrix}.
\label{eq:SingerC}
\end{equation}
Here, $d$ denotes the distance of the vehicle from the preceding vehicle, $v$ is its speed -- both computed from lidar measurements and transmitted over the bus, $a$ is the acceleration, $u$ is the control input
, while $\tau=0.8$ represents maneuver time constant of the preceding vehicle~\cite{ploeg2015graceful}
. }
We focus on the cases when the attacker compromises all 
 car sensors, making the system perfectly attackable. We set maximal estimation error to be $0.5m$ on position, $3.3\frac{m}{s}$ on speed, and $0.3\frac{m}{s^2}$ on acceleration, resulting in 
  $\|\Delta\mathbf{e}_{max}\|_{2}=3.351$. 

As in trajectory tracking, we 
 assume $\varepsilon=0.1\%$,
 and $\beta=0.35\%$. 
Since 
{observability index $\psi=2$} and number of unstable eigenvalues of $\mathbf{A}$ is 2, then $f=2$. 
 For periodic policy with $L=20$ we obtain the maximal reachable estimation errors in the presence of stealthy attacks as presented in \figref{fig:V2VNew}. In addition, visual representation of reachable regions with this policy in comparison to a system without integrity enforcement is shown in~\figref{fig:V2VNew_3D}.
These results illustrate that even with 10\% authenticated messages the system ensures satisfiable performance under false-date injection attacks.  

\begin{figure}[!t]%
	\centering
	\includegraphics[width=0.446\textwidth]{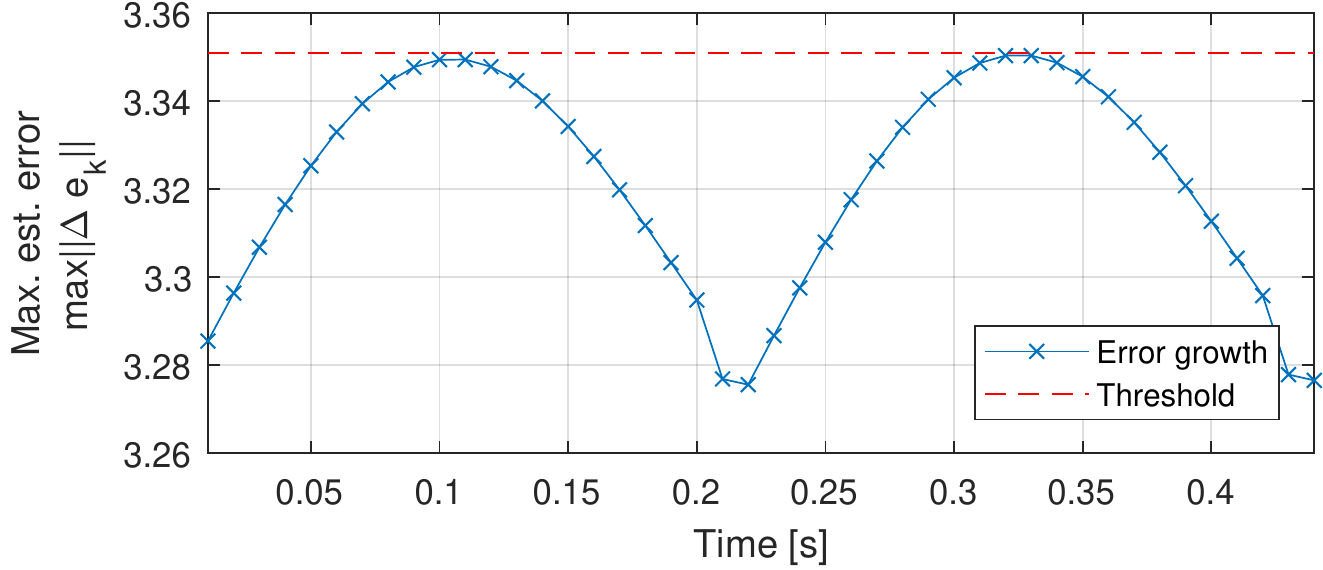}
	\caption{Evolution of maximal estimation error for dCACC
	. If we can enforce integrity on two sensor values after every twenty unsecured sensor values, 
	 the system remains under the specified safety threshold.}
	\label{fig:V2VNew}%
\end{figure}
\begin{figure}[!t]%
	\centering
	\includegraphics[width=0.5\textwidth]{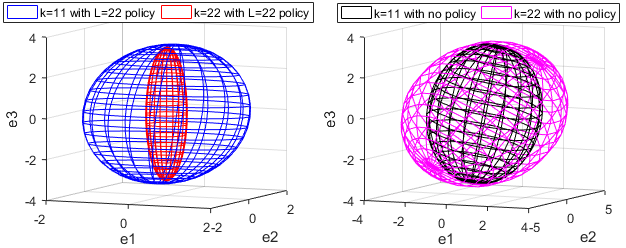}
	\caption{Reachable state estimation errors in the presence of stealthy attacks for dCACC 
	 in steps $k=11$ and $k=22$ with and without data integrity enforcement. Without integrity enforcement, the size of reachable regions keeps increasing, while when integrity is being enforced with policy $L=20$ and $f=2$, estimation error evolves as in \figref{fig:V2VNew}, and the attacker is contained between red and blue ellipsoids.}
	\label{fig:V2VNew_3D}%
\end{figure}

\section{Conclusion}
\label{sec:conclusion}

In this paper, we have focused on the problem of network-based attacks on standard linear state estimators. We have considered systems with Kalman filter-based estimators and a general type of residual-based intrusion detectors, covering widely used detectors such as~$\chi^2$ and SPRT. For these systems, we have studied effects of intermittent data integrity enforcements, such as the use of message authentication codes, on control performance in the presence of attacks. We have shown that when integrity of sensor measurements is enforced only intermittently, a stealthy attacker cannot insert an unbounded state estimation error. In addition, we have introduced a framework that facilitates both evaluation and design of these intermittent policies by providing analysis of the reachable state estimation errors in the presence of stealthy attacks. Although the framework has been developed for systems that employ SPRT detectors, the presented techniques can be extended for detectors from the general class described in Section~\ref{sec:problem}. 
%
Finally, on three automotive case studies, we have highlighted how devastating stealthy false-data injection attacks can be, and how with the use of intermittent integrity enforcement we can ensure desired control performance with a significant reduction in the communication and computation~overhead. 

The presented method to analyze the effects of intermittent use of authentication 
can also provide the foundation for optimal resource allocation in systems where several control loops share communication and computation resources. Although we present some initial results in~\cite{lesi_rtss17} for bandwidth allocation over a shared network, a more systematic approach to optimal resource allocation with strong Quality-of-Control guarantees in the presence of attacks is an avenue for future~work.
\bibliographystyle{abbrv}
\small
\bibliography{bibliography,MyPapersMP}  
{
\appendix[Proof of Lemma~\ref{lemma:bound}]
\label{appendix}

\begin{proof}
From~\eqref{eq:dEk} and~\eqref{eq:dZk}, the system $\Xi$ can be described as
\begin{equation}
	\Delta \mathbf{e}_{k} = \mathbf{A}\Delta \mathbf{e}_{k-1} - \mathbf{K}\Delta \mathbf{z}_{k}
\end{equation}
Thus, due to the stealthiness constraint~\eqref{eqn:z_th}, from perspective of estimation error $\Delta \mathbf{e}_{k}$ the system $\Xi$ is effectively  an unstable system with bounded input $\Delta\mathbf{z}_{k}$. To show that when the estimation error  becomes unbounded, 
the unbounded parts of the vector would belong to vector subspaces corresponding to unstable modes of $\mathbf{A}$ we start by capturing $\Delta \mathbf{e}_{k}$ in a non-recursive form as
\begin{equation}
\label{eq:BoundLemmaDecReqs}
	\Delta\mathbf{e}_{k} =  - \sum_{i=0}^{k-1} \mathbf{A}^{i}\mathbf{K}\Delta\mathbf{z}_{k-i}, 
\end{equation}
since $\Delta\mathbf{e}_0=\mathbf{0}$. 
Also, since eigenvectors and generalized eigenvectors $\mathbf{v}_1,\dots,\mathbf{v}_n$ of $\mathbf{A}$  span $\mathbb{R}^n$, we can decompose the estimation error~as
	\begin{equation}
	\Delta \mathbf{e}_{k} = 
	\sum_{i=1}^n\alpha_i\mathbf{v}_i,\qquad \alpha_i\in\mathbb{R}, i=1,... , n.
	\label{eq:e_vs}
	\end{equation}
	%
Decomposing $\mathbf{K}\Delta\mathbf{z}_{k-i}$ with the same base vectors $\mathbf{v}_1,...\mathbf{v}_n\in\mathbb{R}^n$, we obtain that
	\begin{equation}
		-\mathbf{K}\Delta\mathbf{z}_{k-i} = \phi_{k-i,1}\mathbf{v}_1 + \dots + \phi_{k-i,n}\mathbf{v}_n, \qquad \phi_{k-i,j}\in\mathbb{R},
	\label{eq:z_vs}
	\end{equation}	 
where for $i=0,...k-1$ and $j=1,...n$, due to~\eqref{eqn:z_th} it holds that 
$\phi_{k-i,j}\in\mathbb{R}$ are bounded -- specifically, for $\mathbf{V}=\left[ \mathbf{v}_1 \dots \mathbf{v}_n \right]\in\mathbb{R}^{n\times n}$ 
$$|\phi_{k-i,j}|\leq \phi_{max,i} = \max_{\|\mathbf{z}\|_2=1} \mathbf{i}_{i}^T\mathbf{V}^{-1}(\mathbf{Kz}),$$
where 
$\mathbf{V}$ is invertible as  $\mathbf{v}_1,...,\mathbf{v}_n$ are linearly~independent, and $\mathbf{i}_i\in\mathbb{R}^{n}$ is the projection vector with 1 in $i^{th}$ position and zeros~otherwise. 

Thus, from~\eqref{eq:BoundLemmaDecReqs},~\eqref{eq:e_vs}, and~\eqref{eq:z_vs}, we have
	\begin{equation}
		\sum_{i=1}^n\alpha_i\mathbf{v}_i = \sum_{j=0}^{k-1}\sum_{i=1}^n \phi_{k-j,i}\mathbf{A}^{j-1}\mathbf{v}_i.
		\label{eq:SumOfSum}
	\end{equation}
We now consider two cases, although Case II is more general (and captures Case I as well), its notation is quite cumbersome. 

\vspace{4pt}\noindent
\textbf{\textit{Case~I}} -- When $\mathbf{A}$ is diagonizable, $\mathbf{Av}_i = \lambda_i\mathbf{v}_i$ holds, and thus
	\begin{equation}
		\sum_{i=1}^n\alpha_i\mathbf{v}_i = \sum_{j=0}^{k-1}\sum_{i=1}^n \phi_{k-j,i}\lambda_i^{j-1}\mathbf{v}_i.
		\label{eq:SumOfSum_CaseIa}
	\end{equation}
	Since $\mathbf{v}_i$ are linearly independent, we obtain
	\begin{equation}
		|\alpha_i| = |\sum_{j=0}^{k-1} \phi_{k-j,i}\lambda^{j-1}| \leq \phi_{max,i}\sum_{j=0}^{k-1}|\lambda_i|^{j-1},
		\label{eq:SumOfSum_CaseIb}
	\end{equation}
The right side of~\eqref{eq:SumOfSum_CaseIb} will converge when $k\rightarrow\infty$ if and only if $|\lambda_i|<1$, 
which implies that $\alpha_i$ can have arbitrarily large values only if associated with an eigenvector corresponding to an unstable eigenvalues, while $\alpha_i$ associated with stable eigenvalues is~bounded. 

\vspace{4pt}\noindent	
\textbf{\textit{Case~II}} -- 
In general $\mathbf{A}$ may not be diagonizable, and we consider generalized eigenvectors. 
Specifically, we index (generalized) eigenvectors such that 
 each eigenvector $\mathbf{v}_{i}$ with generalized eigenvectors, 
$\mathbf{v}_{i+1},\dots,\mathbf{v}_{i+L_i}$ 
form its generalized eigenvector chain of length $L_i$ -- i.e.,~for $0\leq l \leq L_i$, $\mathbf{v}_{i+l}$ represents $l$-th element of the chain.
	%
	By representing $\mathbf{A} = \mathbf{VJV}^{-1}$ and 
	$\Delta\mathbf{e}_k = \mathbf{V}[\alpha_1~\dots~\alpha_n]^T$,
	 where $\mathbf{V}=[\mathbf{v}_1~\dots~\mathbf{v}_n]$ 
	 and $\mathbf{J}$ is Jordan form of $\mathbf{A}$, we can exploit 
the property of Jordan block matrices~\cite{golub2012matrix}, 
to obtain following expression
	%
	$$ \mathbf{A}^j\mathbf{v}_{i+L_i} = \sum_{l=0}^{min(j,L_i)} \binom{j}{l}\lambda_{i}^{j-l}\mathbf{v}_{i+L_i-l}. $$
	This allows us to represent~\eqref{eq:SumOfSum} as
	\begin{equation}
		\sum_{i=1}^n\alpha_i\mathbf{v}_i = \sum_{j=0}^{k-1}\sum_{i=1}^n\sum_{l=0}^{min(j,L_i)} \phi_{k-j,i}\binom{j}{l}\lambda_i^{j-l}\mathbf{v}_{i+L_{i}-l},
	\end{equation}
	where $L_{i}$ depends on the particular $i$ that is being summed over. 
	Let $L_{fol(i)}$ denote the number of followers of $\mathbf{v}_i$ inside its eigenvector chain (e.g.,~if $v_i$ is an eigenvector $L_{fol(i)}=L_i$). 
Again, since $\mathbf{v}_i$ are linearly independent, we obtain 

	\begin{multline}
		|\alpha_i| = |\sum_{j=0}^{k-1}\sum_{l=0}^{min(j,L_{fol(i)})} \phi_{k-j,i+l}\binom{j}{l}\lambda_i^{j-l}| \leq \\
		\leq \sum_{j=0}^{k-1}\sum_{l=0}^{min(j,L_{fol(i)})} |\phi_{k-j,i+l}|\binom{j}{l}|\lambda_i|^{j-l} \\
		\leq \sum_{j=0}^{k-1}\sum_{l=0}^{min(j,L_{fol(i)})} \phi_{max} j^{L_{fol(i)}}|\lambda_i|^{j-l} \\
		\leq \left\{ \begin{array}{l} 
			(L_{fol(i)}+1)\phi_{max}\sum_{j=0}^{k-1} j^{L_{fol(i)}}|\lambda_i|^{j}, ~~~~~~~~~~~ |\lambda_i|\geq 1\\
			(L_{fol(i)}+1)\phi_{max}\sum_{j=0}^{k-1} j^{L_{fol(i)}}|\lambda_i|^{j-L_{fol(i)}},~~  |\lambda_i|< 1 \end{array}\right.
		\label{eq:LemmaFinInequality}
	\end{multline}
	where $\phi_{max}=\max\{\phi_{max,1},\dots,\phi_{max,n}\}$. 
	%
	If we use the ratio test for convergence of series~\cite{boas2006mathematical} when $|\lambda_i|<1$, we obtain
	\begin{multline*}	
	\lim_{j\rightarrow\infty}\frac{(j+1)^{L_{fol(i)}}|\lambda_i|^{j+1-L_{fol(i)}}}{j^{L_{fol(i)}}|\lambda_i|^{j-L_{fol(i)}}} = \\
	= \lim_{j\rightarrow\infty}|\lambda_i|(\frac{j+1}{j})^{L_{fol(i)}} = |\lambda_i|.
	\end{multline*}
	Thus, since $|\lambda_i|<1$ by assumption, the series converges, and all $\alpha_i$ that correspond to stable eigenvalues have to be bounded. Similarly, the ratio test can be used to show that the series is divergent when $|\lambda_i|>1$. 
	Divergence of series for $|\lambda_i|=1$ can be shown by substitution. Namely, from~\eqref{eq:LemmaFinInequality}, when $|\lambda_i|=1$, it follows that	
	%
	$$ (L_{fol(i)}+1)\phi_{max}\sum_{j=0}^{k-1} j^{L_{fol(i)}}|\lambda_i|^{j} = (L_{fol(i)}+1)\phi_{max}\sum_{j=0}^{k-1} j^{L_{fol(i)}},$$
%
which given that $L_{fol(i)}\in\mathbb{N}_0$, implies that the series also diverges for $|\lambda_i|=1$, and thus concludes the proof. 
\end{proof}
}

\end{document}